\documentclass[12pt]{article}

\usepackage{a4wide,epsfig}
\usepackage{amsfonts,amsmath}
\usepackage{subcaption}
\usepackage{algorithm}
\usepackage{algpseudocode}
\usepackage{pgfplotstable}

\newtheorem{theorem}{Theorem}[section]

\newtheorem{corollary}[theorem]{Corollary}
\newtheorem{lemma}[theorem]{Lemma}

\newtheorem{remark}{Remark}[section]
\newenvironment{proof}{\textbf{Proof.}}{\hfill $\Box$}

\numberwithin{equation}{section} 

\begin{document}

\setcounter{page}{1}

\title{Regularization and finite element error estimates for elliptic
  distributed optimal control problems with energy regularization and
state or control constraints}
\author{Peter~Gangl$^1$, Richard~L\"oscher$^2$, Olaf~Steinbach$^2$}
\date{$^1$Johann Radon Institute for Computational and
  Applied Mathematics, \\
Altenberger Stra{\ss}e 69, 4040 Linz, Austria \\[2mm]
$^2$Institut f\"ur Angewandte Mathematik, TU Graz, \\[1mm] 
Steyrergasse 30, 8010 Graz, Austria}

\maketitle

\begin{abstract}
  In this paper we discuss the numerical solution of elliptic distributed
optimal control problems with state or control constraints when the
control is considered in the energy norm. As in the unconstrained case
we can relate the regularization parameter and the finite element mesh
size in order to ensure an optimal order of convergence which only
depends on the regularity of the given target, also including
discontinuous target functions. While in most cases, state or control
constraints are discussed for the more common $L^2$ regularization,
much less is known in the case of energy regularizations.
But in this case, and for both control and state constraints, we can
formulate first kind variational inequalities to determine the
unknown state, from wich we can compute the control in a post
processing step. Related variational inequalities also appear in
obstacle problems, and are well established both from a mathematical
and a numerical analysis point of view. Numerical results confirm
the applicability and accuracy of the proposed approach.

\end{abstract}

\section{Introduction}
Optimal control problems aim to determine a control of a system that drives
the corresponding state as closely as possible to a given desired state
under acceptable costs for the control,
see \cite{Troeltzsch_book} for a thorough overview of the mathematical
theory. Over the past decades, such problems have been studied for a wide
variety of applications. A prominent example is the medical application of
cancer treatment by hyperthermia \cite{DeuflhardSchielaWeiser:2012} where a
heat source should be placed in such a way that the temperature is increased
(only) in the cancerous tissue. The control variable can, in general, be
defined on the full domain or on the boundary, and typically the application
of the control comes at a certain cost which can be measured in different
norms such as the $L^2$ norm or an energy norm. This cost is typically added
to the objective functional with a certain weight and can also be seen as
a regularization of the PDE-constrained optimization problem.
Finite element error estimates of solutions to optimal control problems
have been studied by many authors, see, e.g., \cite{Karkulik:2020}, for an
elliptic boundary control problem or
\cite{ApelSteinbachWinkler:2016, Chowdhury:2017, Gong:2022,
  OfPhanSteinbach:2015} for boundary control with energy regularization.
More recently, time-dependent optimal control problems in the context of
space-time finite element methods and corresponding error estimates were
studied by some of the authors, see, e.g.,
\cite{LangerSteinbachTroeltzschYang:2021, LangerSteinbachYang:2022b,
  LangerSteinbachYang:2022a} for parabolic problems or
\cite{LoescherSteinbach:2022} for the wave equation.
For many optimal control problems, it is important to pose pointwise
constraints for either the state or the control variable, or both. These
constraints can be incorporated in different ways, e.g., by augmented
Lagrangian methods \cite{KarlWachsmuth:2018}, barrier methods
\cite{SchielaWeiser:2010} or by reformulating the optimality system as a
variational inequality \cite{BrezziHagerRaviart:1977, Glowinski:1980,
  LionsStampacchia:1967}. This is closely related to the treatment of
obstacle problems, where constraints can be handled using a penalization
technique, see, e.g., \cite{Kikuchi:1981}.
After discretization, variational inequalities can be solved by means of
primal-dual active set strategies \cite{BergouniouxKunisch:2002} or
semi-smooth Newton methods where the latter two strategies can be shown to
be equivalent \cite{Hintermueller:2002}. In particular the latter class
has been used in different physical contexts such as elasticity
\cite{Kroener:2013}, fluid mechanics \cite{ReyesKunisch:2005}, wave
problems \cite{KroenerKunischVexler:2011}, and for different kinds of
constraints including mixed control-state constraints
\cite{RoeschWachsmuth:2011} or constraints on derivatives of the
state \cite{HintermuellerKunisch:2009}. As in the unconstrained case, the
control can be measured in different norms, depending on the regularity
assumptions on the control, which leads to a different behavior of the
solutions in particular in the case of less regular, i.e.,
discontinuous targets.
The, nowaday, common $L^2$ regularization with state constraints
was already studied in \cite{Falk:1973}. Considering the $L^2$ norm as
energy norm leads to a fourth order elliptic partial differential
equation, see \cite{NeumuellerSteinbach:2021}. The recent survey paper
\cite{Brenner:2020} gives an overview on the numerical analysis
incorporating state constraints in this case. While in most cases,
state or control constraints are discussed in the context of
$L^2$ regularizations, much less is known in the case of energy
regularizations.
The very recent work \cite{GongTan:2023} examines state constraints in the
case of energy regularization for the Laplace equation, which is also a
starting point for our analysis, see \cite{Steinbach:2023}.

In this paper, we consider the problem to find the minimizer
$(u_\varrho , z_\varrho) \in H^1_0(\Omega) \times H^{-1}(\Omega)$
of the functional
\begin{equation}\label{cost functional}
  {\mathcal{J}}(u_\varrho,z_\varrho) :=
  \frac{1}{2} \, \| u_\varrho - \overline{u} \|^2_{L^2(\Omega)} +
  \frac{1}{2} \, \varrho \, \| z_\varrho \|^2_{H^{-1}(\Omega)} 
\end{equation}
subject to the Dirichlet boundary value problem of the Poisson
equation,
\begin{equation}\label{DBVP}
  - \Delta u_\varrho = z_\varrho \quad \mbox{in} \; \Omega, \quad
  u_\varrho = 0 \quad \mbox{on} \; \partial \Omega .
\end{equation}
Here, $\Omega \subset {\mathbb{R}}^n$, $n=1,2,3$, is some bounded
Lipschitz domain, $\overline{u} \in L^2(\Omega)$ is a given target,
and $\varrho \in {\mathbb{R}}_+$ is some regularization parameter
on which the minimizer depends on. The variational formulation of the
Dirichlet boundary value problem \eqref{DBVP} is to find
$u_\varrho \in H^1_0(\Omega)$ such that
\begin{equation}\label{DBVP VF}
  \langle \nabla u_\varrho , \nabla v \rangle_{L^2(\Omega)} =
  \langle z_\varrho , v \rangle_\Omega \quad
  \mbox{for all} \; v \in H^1_0(\Omega),
\end{equation}
where $\langle z_\varrho , v \rangle_\Omega$ denotes the duality pairing
for $z_\varrho \in H^{-1}(\Omega) = [H^1_0(\Omega)]^*$ and
$v \in H^1_0(\Omega)$ as extension
of the inner product in $L^2(\Omega)$. Recall that
$\| \nabla v \|_{L^2(\Omega)}$ defines an equivalent norm for
$v \in H^1_0(\Omega)$, and that the dual norm in
$H^{-1}(\Omega)$ is given by
\[
  \| z \|_{H^{-1}(\Omega)} =
  \sup\limits_{0 \neq v \in H^1_0(\Omega)}
  \frac{\langle z , v \rangle_\Omega}{\| \nabla v \|_{L^2(\Omega)}}
  \quad \mbox{for all} \; z \in H^{-1}(\Omega) .
\]
With this we easily conclude
\[
  \| z_\varrho \|_{H^{-1}(\Omega)}^2 =
  \| \nabla u_\varrho \|_{L^2(\Omega)}^2 =
  \langle z_\varrho , u_\varrho \rangle_\Omega ,
\]
where $u_\varrho \in H^1_0(\Omega)$ is the unique solution of the
variational formulation \eqref{DBVP VF}.
Hence we can write the cost functional \eqref{cost functional}
as reduced cost functional
\begin{equation}\label{reduced cost functional}
  \widetilde{\mathcal{J}}(u_\varrho) =
  \frac{1}{2} \, \| u_\varrho - \overline{u} \|^2_{L^2(\Omega)} +
  \frac{1}{2} \, \varrho \, \| \nabla u_\varrho \|^2_{L^2(\Omega)} .
\end{equation}
In the case of neither state nor control constraints, the minimizer
of the reduced cost functional \eqref{reduced cost functional} is
given as the unique solution $u_\varrho \in H^1_0(\Omega)$ of the
variational formulation
\begin{equation}\label{gradient equation VF}
  \varrho \, \langle \nabla u_\varrho , \nabla v \rangle_{L^2(\Omega)} +
  \langle u_\varrho , v \rangle_{L^2(\Omega)} =
  \langle \overline{u} , v \rangle_{L^2(\Omega)} \quad
  \mbox{for all} \; v \in H^1_0(\Omega) .
\end{equation}
Depending on the regularity of the target $\overline{u}$ we can prove
the following regularization error estimates:

\begin{lemma}[{\cite[Theorem 3.2]{NeumuellerSteinbach:2021}}]
  For $\varrho > 0$,
  let $u_\varrho \in H^1_0(\Omega)$ be the unique solution of the
  variational formulation \eqref{gradient equation VF}.
  Assume $\overline{u} \in H^s_0(\Omega) := [L^2(\Omega),H^1_0(\Omega)]_s$
  for some $s \in [0,1]$ or $\overline{u} \in H^1_0(\Omega) \cap H^s(\Omega)$
  for some $s \in (1,2]$. Then there holds the regularization error estimate
  \begin{equation}\label{regularization error estimate}
    \| u_\varrho - \overline{u} \|_{L^2(\Omega)} \leq c \, \varrho^{s/2} \,
    \| \overline{u} \|_{H^s(\Omega)} .
  \end{equation}
\end{lemma}

\noindent
Let $V_h \subset H^1_0(\Omega)$ be some finite element space of piecewise
linear and continuous basis functions which are defined with respect
to some admissible decomposition of the domain $\Omega$
into simplicial shape regular
finite elements $\tau_\ell$ of local mesh size $h_\ell$,
and with a global mesh size $h:= \max_\ell h_\ell$. For simplicity we may
assume that $\Omega$ is a polygonally ($n=2$) or polyhedrally ($n=3$)
bounded domain.
The finite element approximation of \eqref{gradient equation VF} is
to find $u_{\varrho h} \in V_h$ such that
\begin{equation}\label{gradient equation FEM}
  \varrho \, \langle \nabla u_{\varrho h}, \nabla v_h \rangle_{L^2(\Omega)} +
  \langle u_{\varrho h} , v_h \rangle_{L^2(\Omega)} =
  \langle \overline{u} , v_h \rangle_{L^2(\Omega)} \quad
  \mbox{for all} \; v_h \in V_h .
\end{equation}
The numerical analysis of this variational formulation as well as
the construction of robust iterative solution methods was considered
in \cite{LangerSteinbachYang:2022a}.

\begin{lemma}[{\cite[Theorem 1]{LangerSteinbachYang:2022a}}]
  Let us assume, for simplicity, that $\Omega \subset {\mathbb{R}}^n$
  is convex, and that the target function satisfies either
  $\overline{u} \in H^s_0(\Omega)$ for $s \in [0,1]$ or
  $\overline{u} \in H^1_0(\Omega) \cap H^s(\Omega)$ for $s \in (1,2]$.
  For the choice $\varrho = h^2$ there holds the error estimate
  \begin{equation}\label{FEM error estimate}
    \| u_{\varrho h} - \overline{u} \|_{L^2(\Omega)} \leq
    c \, h^s \, \| \overline{u} \|_{H^s(\Omega)} .
  \end{equation}
\end{lemma}

\noindent
The aim of this paper is to provide related estimates for both
the regularization and the finite element error when considering
the minimization of \eqref{reduced cost functional} with
either state or control constraints. Note that related work, considering
state constraints, but not with respect to the regularization parameter
$\varrho$, was recently presented in \cite{GongTan:2023}.

The remainder of this paper is structured as follows: In
Section~\ref{Chapter:Regularization} we provide 
estimates for the error $\| u_\varrho - \overline{u} \|_{L^2(\Omega)}$
with respect to the regularization parameter $\varrho$ for both state
and control constraints. These results follow similar as in the
unconstrained case, due to the structure of the variational inequality
to be solved. In order to include state or control constraints we
have to solve a first kind variational inequality to find the unknown
state. Their numerical
approximation using finite element methods is well established, 
and we can transfer the general results to the particular application
of constrained optimal control problems with energy regularization.
The finite element
discretization and the related a priori error estimates are given
in Section~\ref{Chapter FEM}. In a post processing step we finally
compute the control, when the state is known. The resulting discrete
variational inequality can be reformulated as a nonlinear system
of algebraic equations, which can be solved by applying a semi-smooth
Newton method which turns out to be an active set strategy, and
which is described in Section~\ref{Section:Newton}. Several numerical
results are given in Section~\ref{Section:Results} in order to
demonstrate the applicability and accuracy of the proposed approach.
Finally, we summarize and comment on ongoing work.

\section{Regularization error estimates}\label{Chapter:Regularization}
In this section we will discuss regularization error estimates for the
minimization of \eqref{reduced cost functional} with additional
constraints on either the state $u_\varrho$ or the control $z_\varrho$.

\subsection{State constraints}
We consider the reduced functional $\widetilde{\mathcal{J}}(u_\varrho)$
as defined in \eqref{reduced cost functional}, but now we minimize
over the convex subset
\[
  K_s := \Big \{ v \in H^1_0(\Omega) : g_-(x) \leq v(x) \leq g_+(x) \;
  \mbox{for almost all} \; x \in \Omega \Big \},
\]
where
$ g_{\pm} \in H^1_{\Delta}(\Omega):=\{v\in H^1_0(\Omega):\, \Delta v \in L^2(\Omega)\}$ are given barrier functions, and where we assume 
$g_- < g_+$ and $0 \in K_s$ to be satisfied. From this it follows
that $g_- \leq 0$, and $g_+ \geq 0$. The minimizer $u_\varrho \in K_s$
satisfying
\[
  \widetilde{\mathcal{J}}(u_\varrho) = \min\limits_{v \in K_s}
  \widetilde{\mathcal{J}}(v)
\]
is determined as the unique solution $u_\varrho \in K_s$ of the
first kind variational inequality
\begin{equation}\label{VI state constraints}
  \varrho \, \langle \nabla u_\varrho , \nabla (v-u_\varrho)
  \rangle_{L^2(\Omega)} +
  \langle u_\varrho , v - u_\varrho \rangle_{L^2(\Omega)} \geq
  \langle \overline{u} , v - u_\varrho \rangle_{L^2(\Omega)} \quad
  \mbox{for all} \; v \in K_s .
\end{equation}

\noindent
As in the unconstrained case \cite{NeumuellerSteinbach:2021} we can prove the
following regularization error estimates.

\begin{lemma}\label{Lemma regularization state constraints}
  For $\varrho > 0$,
  let $u_\varrho \in K_s$ be the unique solution of the variational
  inequality \eqref{VI state constraints}. For $\overline{u} \in L^2(\Omega)$
  there holds
  \begin{equation}\label{regularization state constraints H H}
    \| u_\varrho - \overline{u} \|_{L^2(\Omega)} \leq
    \| \overline{u} \|_{L^2(\Omega)} ,
  \end{equation}
  while for $\overline{u} \in K_s$ we have
  \begin{equation}\label{regularization state constraints H S}
    \| u_\varrho - \overline{u} \|_{L^2(\Omega)} \leq
    \sqrt{\varrho} \, \| \nabla \overline{u} \|_{L^2(\Omega)},
  \end{equation}
  and
  \begin{equation}\label{regularization state constraints S S}
    \| \nabla (u_\varrho - \overline{u}) \|_{L^2(\Omega)} \leq
    \| \nabla \overline{u} \|_{L^2(\Omega)} .
  \end{equation}
  If in addition $\Delta \overline{u} \in L^2(\Omega)$ is satisfied for
  $\overline{u} \in K_s$,
  \begin{equation}\label{regularization state constraints H Su}
    \| u_\varrho - \overline{u} \|_{L^2(\Omega)} \leq
    \varrho \, \| \Delta \overline{u} \|_{L^2(\Omega)}
  \end{equation}
  as well as
  \begin{equation}\label{regularization state constraints S Su}
    \| \nabla (u_\varrho - \overline{u}) \|_{L^2(\Omega)} \leq
    \sqrt{\varrho} \, \| \Delta \overline{u} \|_{L^2(\Omega)}
  \end{equation}  
  follow.
\end{lemma}

\begin{proof}
From the variational inequality \eqref{VI state constraints}
we obviously have
\[
  \varrho \, \langle \nabla u_\varrho ,
  \nabla (v - u_\varrho) \rangle_{L^2(\Omega)} \geq
 \langle \overline{u} - u_\varrho , v - u_\varrho \rangle_{L^2(\Omega)} \quad
 \mbox{for all} \; v \in K_s .
\]
In particular for $v = 0 \in K_s$ this gives
\[
 \| \overline{u} - u_\varrho \|^2_{L^2(\Omega)}
 +
 \varrho \, \langle \nabla u_\varrho , \nabla u_\varrho \rangle_{L^2(\Omega)}
 \leq
 \langle u_\varrho - \overline{u} , \overline{u} \rangle_{L^2(\Omega)}
 \leq
 \| \overline{u} - u_\varrho \|_{L^2(\Omega)} \| \overline{u} \|_{L^2(\Omega)},
\]
i.e., \eqref{regularization state constraints H H} follows.
When assuming $\overline{u} \in K_s$ we can consider
$v=\overline{u}$ to obtain
\begin{eqnarray*}
  \varrho \, \| \nabla (u_\varrho - \overline{u}) \|^2_{L^2(\Omega)} +
  \| u_\varrho - \overline{u} \|_{L^2(\Omega)}^2
  & \leq &  
  \varrho \, \langle \nabla \overline{u} ,
  \nabla (\overline{u} - u_\varrho) \rangle_{L^2(\Omega)} \\
  & \leq &  \varrho \, \| \nabla \overline{u} \|_{L^2(\Omega)}
  \| \nabla (u_\varrho - \overline{u}) \|_{L^2(\Omega)} ,
\end{eqnarray*}
i.e.,
\[
  \| \nabla (u_\varrho - \overline{u}) \|_{L^2(\Omega)} \leq
  \| \nabla \overline{u} \|_{L^2(\Omega)},
\]
that is \eqref{regularization state constraints S S},
and \eqref{regularization state constraints H S},
\[
  \| u_\varrho - \overline{u} \|_{L^2(\Omega)}^2 \leq
  \varrho \, \| \nabla \overline{u} \|_{L^2(\Omega)}^2 \, .
\]
If $\overline{u} \in K_s$ is such that $\Delta \overline{u} \in L^2(\Omega)$
is satisfied, then we conclude
\begin{eqnarray*}
  \varrho \, \| \nabla (u_\varrho - \overline{u}) \|^2_{L^2(\Omega)} +
  \| u_\varrho - \overline{u} \|_{L^2(\Omega)}^2
  & \leq & \varrho \, \langle \nabla \overline{u} ,
           \nabla (\overline{u} - u_\varrho) \rangle_{L^2(\Omega)} \\
  & & \hspace*{-2cm} = \, \varrho \, \langle (- \Delta \overline{u}) ,
        \overline{u} - u_\varrho \rangle_{L^2(\Omega)} \,
  \leq \, \varrho \, \| \Delta \overline{u} \|_{L^2(\Omega)}
           \| u_\varrho - \overline{u} \|_{L^2(\Omega)} ,
\end{eqnarray*}
and hence
\[
  \| u_\varrho - \overline{u} \|_{L^2(\Omega)} \leq
  \varrho \, \| \Delta \overline{u} \|_{L^2(\Omega)} ,
\]
i.e., \eqref{regularization state constraints H Su}, follows. Finally,
\[
  \varrho \, \| \nabla (u_\varrho - \overline{u}) \|_{L^2(\Omega)}^2 \leq
  \varrho \, \| \Delta \overline{u} \|_{L^2(\Omega)}
  \| u_\varrho - \overline{u} \|_{L^2(\Omega)} \leq
  \varrho^2 \, \| \Delta \overline{u} \|_{L^2(\Omega)}^2
\]
implies \eqref{regularization state constraints S Su}.
\end{proof}

\noindent
For the solution $u_\varrho$ of \eqref{VI state constraints}
we introduce the active sets
$\Omega_{s,\pm} := \{ x \in \Omega : u_\varrho(x) = g_\pm(x) \}$.

\begin{lemma}\label{Lemma 2.2}
  For $u_\varrho \in K_s$ being the unique solution of the variational
  inequality \eqref{VI state constraints}, let
  $\lambda := - \varrho \Delta u_\varrho + u_\varrho - \overline{u}
  \in H^{-1}(\Omega)$. Assume $\overline{u} \in K_s\cap H^1_{\Delta}(\Omega)$ and $g_\pm \in H^1_{\Delta}(\Omega)$. 
  Then we conclude $z_\varrho=-\Delta u_\varrho \in L^2(\Omega)$ and hence,
  $\lambda \in L^2(\Omega)$.
\end{lemma}

\begin{proof}
  When using integration by parts we can write
\eqref{VI state constraints} as
\[
  \langle - \varrho \Delta u_\varrho + u_\varrho - \overline{u} ,
  v-u_\varrho \rangle_\Omega \geq 0 \quad \mbox{for all} \; v \in K_s,
\]
i.e.,
\begin{equation}\label{VI state lambda}
  \langle \lambda , v - u_\varrho \rangle_\Omega \geq 0
  \quad \mbox{for all} \; v \in K_s .
\end{equation}
The definition of $\lambda$ implies
\[
  \lambda + \varrho \Delta u_\varrho = u_\varrho - \overline{u} \in L^2(\Omega).
\]
For $x \in \Omega_\pm$ we have $u_\varrho(x)=g_\pm(x)$, and hence
$\Delta u_\varrho = \Delta g_\pm \in L^2(\Omega_{s,\pm})$ as well as
$\lambda \in L^2(\Omega_{s,\pm})$. Let $ w \in H^1_0(\Omega)$ satisfying 
\[
  0 \leq w(x) \leq \min \Big \{ g_+(x) - u_\varrho(x),
  u_\varrho(x)-g_-(x) \Big \} \quad \mbox{for} \; x \in \Omega ,
\]
i.e., $w(x)=0$ for $x \in \Omega_{s,\pm}$.
For $ v = u_\varrho + w \in K_s$ we then obtain from
\eqref{VI state lambda}
$\langle \lambda , w \rangle_\Omega \geq 0$,
while for $ v = u_\varrho - w \in K_s $ we conclude
$\langle \lambda , w \rangle_\Omega \leq 0$.
Hence we have
$\langle \lambda , w \rangle_\Omega = 0$ for all
$w \in H^1_0(\Omega \backslash \Omega_{s,\pm})$,
i.e., $\lambda = 0$ in $H^{-1}(\Omega \backslash \Omega_{s,\pm})$,
which remains true in $L^2(\Omega \backslash \Omega_{s,\pm})$. This
already gives $\lambda \in L^2(\Omega)$. Moreover,
by $0=\lambda =-\varrho\Delta u_\varrho+u_\varrho-\overline u$ in
$\Omega\setminus \Omega_{s,\pm}$ and
Lemma \ref{Lemma regularization state constraints} we obtain
\[
  \| \varrho \Delta u_\varrho \|_{L^2(\Omega \backslash \Omega_{s,\pm})} =
  \| u_\varrho - \overline{u} \|_{L^2(\Omega \backslash \Omega_{s,\pm})} \leq
  \| u_\varrho - \overline{u} \|_{L^2(\Omega)} \leq \varrho \,
  \| \Delta \overline{u} \|_{L^2(\Omega)},
\]
which implies
\[
  \| \Delta u_\varrho \|_{L^2(\Omega \backslash \Omega_\pm)} \leq
  \| \Delta \overline{u} \|_{L^2(\Omega)},
\]
i.e., $\Delta u_\varrho \in L^2(\Omega\setminus \Omega_\pm)$ and together
with $\Delta u_\varrho = \Delta g_\pm$ in $\Omega_{s,\pm}$,
$\Delta u_\varrho \in L^2(\Omega)$, independent of $\varrho$.
\end{proof}

\noindent
Due to $\lambda \in L^2(\Omega)$ we can write \eqref{VI state lambda} as
\[
  \int_\Omega \lambda(x) \, [v(x)-u_\varrho(x)] \, dx \geq 0
  \quad \mbox{for all} \; v \in K_s .
\]
For arbitrary $w_+ \in H^1_0(\Omega)$ satisfying
$0 \leq w_+(x) \leq g_+(x) - u_\varrho(x)$ for almost all $x \in \Omega$
we have $ v = u_\varrho + w_+ \in K_s$, and we conclude
\[
  \int_{\Omega \backslash \Omega_+} \lambda(x) \, w_+(x) \, ds_x \geq 0 \quad
  \mbox{for all} \; w_+ \in H^1_0(\Omega\backslash \Omega_{s,+}), \quad
  w_+ \geq 0 .
\]
Hence we obtain $\lambda(x) \geq 0$ for almost all
$x \in \Omega \backslash \Omega_{s,+}$.
In the same way we choose $w_- \in H^1_0(\Omega)$ satisfying
$ g_-(x) - u_\varrho(x) \leq w_-(x) \leq 0$ for almost all $x \in \Omega$.
Hence we have $v = u_\varrho + w_- \in K_s$, and we conclude
\[
  \int_{\Omega \backslash \Omega_-} \lambda(x) \, w_-(x) \, dx \geq 0 \quad
  \mbox{for all} \; w_- \in H^1_0(\Omega \backslash \Omega_{s,-}), \quad
  w_- \leq 0 ,
\]
i.e.,
$\lambda(x) \leq 0$ for almost all $x \in \Omega \backslash \Omega_{s,-}$.
With this we finally obtain the complementarity conditions which hold
for almost all $x \in \Omega$:
\[ 
  g_-(x) < u_\varrho(x) < g_+(x) \, : \, \lambda(x)=0 ; \,
  u_\varrho(x) = g_-(x) \, : \, \lambda(x) \geq 0 ; \,
  u_\varrho(x) = g_+(x) \, : \, \lambda(x) \leq 0. 
\]

\begin{remark}
The variational inequality \eqref{VI state constraints} corresponds
to the two obstacle problem as considered, e.g., in
{\rm \cite{BrezisStampacchia:1968}}, where also a more general
discussion on the regularity of solutions is given, i.e., {\rm
  \cite[Theor\'em\`e I.1, Remarque I.4, Remarque I.5]{BrezisStampacchia:1968}},
which also fits our application.
\end{remark}

\subsection{Control constraints}
Since $- \Delta : H^1_0(\Omega) \to H^{-1}(\Omega)$ defines an
isomorphism, we can also write $u_\varrho = S z_\varrho$, where
$S : H^{-1}(\Omega) \to H^1_0(\Omega)$ is the solution operator
of the Dirichlet boundary value problem \eqref{DBVP}. Instead of
\eqref{cost functional} and \eqref{reduced cost functional}
we now consider the reduced cost functional
\begin{equation}\label{reduced cost functional control}
  \widehat{\mathcal{J}}(z_\varrho) =
  \frac{1}{2} \, \| S z_\varrho - \overline{u} \|^2_{L^2(\Omega)} +
  \frac{1}{2} \, \varrho \, \langle S z_\varrho , z_\varrho \rangle_\Omega
  \quad \mbox{for} \; z_\varrho \in H^{-1}(\Omega) .
\end{equation}
Box constraints in $H^{-1}(\Omega)$ are defined in weak form, i.e., for given
$f_{\pm} \in L^2(\Omega)$ we define
\begin{equation}\label{box constraints weak}
  Z_c := \Big \{ z \in H^{-1}(\Omega) : 
             \langle f_- , v \rangle_{L^2(\Omega)} \leq
             \langle z_\varrho , v \rangle_\Omega
             \leq \langle f_+ , v \rangle_{L^2(\Omega)} \;
             \forall \, v \in H^1_0(\Omega), \; v(x) \geq 0 \Big \} .
\end{equation}
Hence we find the minimizer $z_\varrho \in Z_c$ of the reduced cost functional
\eqref{reduced cost functional control} as the unique solution of the
variational inequality
\begin{equation}\label{VI control constraints z}
  \langle S^* S z_\varrho + \varrho \, S z_\varrho ,
  \psi - z_\varrho \rangle_\Omega \geq
  \langle S^* \overline{u} , \psi - z_\varrho \rangle_\Omega
  \quad \mbox{for all} \; \psi \in Z_c .
\end{equation}
When using $u_\varrho = S z_\varrho$ and the fact that $S$ is self-adjoint,
this can be written as
\[
  \langle u_\varrho + \varrho \, z_\varrho ,
  v - u_\varrho \rangle_\Omega \geq
  \langle \overline{u} , v - u_\varrho \rangle_\Omega
  \quad \mbox{for all} \; v = S \psi, \; \psi \in Z_c .
\]
When introducing
\begin{equation}\label{box constraints weak u}
  K_c :=  \Big \{ u \in H^1_0(\Omega) : 
             \langle f_- , v \rangle_{L^2(\Omega)} \leq
             \langle \nabla u , \nabla v \rangle_{L^2(\Omega)}
             \leq \langle f_+ , v \rangle_{L^2(\Omega)} \,
  \forall \, v \in H^1_0(\Omega), \, v(x) \geq 0 \Big \} ,
\end{equation}
and using $z_\varrho = - \Delta u_\varrho$,
we finally end up with
a variational inequality to find $u_\varrho \in K_c$ such that
\begin{equation}\label{VI control constraints u}
  \langle u_\varrho , v - u_\varrho \rangle_{L^2(\Omega)} +
  \varrho \, \langle \nabla u_\varrho , \nabla (v-u_\varrho) \rangle_{L^2(\Omega)}
  \geq
  \langle \overline{u} , v - u_\varrho \rangle_{L^2(\Omega)}
  \quad \mbox{for all} \; v \in K_c .
\end{equation}
Since the variational inequality \eqref{VI control constraints u} coincides
with \eqref{VI state constraints}, all the regularization error estimates
as given in Lemma \ref{Lemma regularization state constraints} remain valid, but
we have to assume $\overline{u} \in K_c$ instead of
$\overline{u} \in K_s$, when required.

For the unique solution $u_\varrho \in K_c$ and for the target
$\overline{u} \in L^2(\Omega)$
we define $w \in H^1_0(\Omega)$ as the unique weak solution of the
Dirichlet boundary value problem
\begin{equation}\label{PDE Def w}
  - \Delta w = - \varrho \Delta u_\varrho + u_\varrho - \overline{u} \quad
  \mbox{in} \; \Omega, \quad w = 0 \quad \mbox{on} \; \partial \Omega,
\end{equation}
satisfying
\[
  \langle \nabla w , \nabla v  \rangle_{L^2(\Omega)} =
  \varrho \, \langle \nabla u_\varrho , \nabla v \rangle_{L^2(\Omega)}
  + \langle u_\varrho - \overline{u} , v \rangle_{L^2(\Omega)}
  \quad \mbox{for all} \; v \in H^1_0(\Omega).
\]
When using integration by parts we can write this as
\[
  \langle - \Delta w + \varrho \Delta u_\varrho , v \rangle_\Omega =
  \langle u_\varrho - \overline{u} , v \rangle_{L^2(\Omega)}
  \quad \mbox{for all} \; v \in H^1_0(\Omega) .
\]
For $f_\pm \in L^2(\Omega)$ we define $g_{c,\pm} \in H^1_0(\Omega)$
as unique solutions of the Dirichlet boundary value problems
\[
  - \Delta g_{c,\pm} = f_\pm \quad \mbox{in} \; \Omega, \quad
  g_{c,\pm} = 0 \quad \mbox{on} \; \partial \Omega, 
\]
and we introduce
$\Omega_{c,\pm} := \{ x\in \Omega : u_\varrho(x) = g_{c,\pm}(x) \}$. Due to
$f_\pm \in L^2(\Omega)$ we therefore have
$f_\pm(x) = -  \Delta g_{c,\pm}(x) = - \Delta u_\varrho(x)$ for almost all
$x \in \Omega_{c,\pm}$.

\begin{lemma}\label{Lemma 2.3}
  For $u_\varrho \in K_c$ being the unique solution of the variational
  inequality \eqref{VI control constraints u}, let $w \in H^1_0(\Omega)$
  be the weak solution of the Dirichlet boundary value problem
  \eqref{PDE Def w}. Assume $\overline{u} \in K_c\cap H^1_{\Delta}(\Omega)$ and $f_\pm \in L^2(\Omega)$.
  Then we conclude $z_\varrho=-\Delta u_\varrho \in L^2(\Omega)$, and hence,
  $\Delta w \in L^2(\Omega)$.
\end{lemma}

\begin{proof}
  The definition of $w \in H^1_0(\Omega)$ as weak solution of
  the Poisson equation in \eqref{PDE Def w} implies
  \[
    - \Delta w(x) = \varrho \, f_\pm(x) + u_\varrho(x) - \overline{u}(x) \quad
    \mbox{for almost all} \; x \in \Omega_{c,\pm},
  \]
  and hence, $\Delta w \in L^2(\Omega_{c,\pm})$ follows.
  Since $u_\varrho \in K_c$ is the unique solution of the variational
  inequality \eqref{VI control constraints u}, and using the definition
  of $w \in H^1_0(\Omega)$, this gives
  \[
    \langle \nabla w , \nabla (v - u_\varrho) \rangle_{L^2(\Omega)} \geq 0
    \quad \mbox{for all} \; v \in K_c,
  \]
  or equivalently,
  \[
    \langle w , - \Delta v + \Delta u_\varrho \rangle_\Omega \geq 0
    \quad \mbox{for all} \; v \in K_c .
  \]
  Let $ v_+ \in H^1_0(\Omega)$ be the unique solution of the
  Dirichlet boundary value problem
  \[
    - \Delta v_+ = - \Delta u_\varrho + \psi \quad
    \mbox{in} \; \Omega, \quad v_+ = 0 \quad \mbox{on} \; \partial \Omega ,
  \]
  where $\psi \in L^2(\Omega)$, $\psi(x) \geq 0$ for almost all
  $x \in \Omega$, is given. To ensure $v_+ \in K_c$  we need to assume
  \[
    \langle \psi , v \rangle_{L^2(\Omega)} \leq
    \langle f_+ + \Delta u_\varrho , v \rangle_\Omega \quad
    \mbox{for all} \; v \in H^1_0(\Omega), \; v \geq 0 .
  \]
  From this we conclude, when considering $v \in H^1_0(\Omega_{c,+})$,
  $\psi(x) = 0$ for almost all $x \in \Omega_{c,+}$, and 
  \begin{equation}\label{w psi +}
    \langle w , \psi \rangle_{L^2(\Omega)} = \langle w , \psi \rangle_{L^2(\Omega\setminus\Omega_{c,+})} \geq 0 \, .
  \end{equation}
  Next, and using the same $\psi$ as above,
  let $v_- \in H^1_0(\Omega)$ be the unique solution of the
  Dirichlet boundary value problem
  \[
    - \Delta v_- = - \Delta u_\varrho - \psi \quad
    \mbox{in} \; \Omega, \quad v_- = 0 \quad \mbox{on}\;
    \partial \Omega.
  \]
  To ensure $v_- \in K_c$ we now have to satisfy
  \[
    \langle \psi , v \rangle_\Omega \leq
    - \langle f_- + \Delta u_\varrho , v \rangle_\Omega
    \quad \mbox{for all} \;
    v \in H^1_0(\Omega), \; v \geq 0 .
  \]
  This gives $\psi(x) = 0$ for almost all $x \in \Omega_{c,-}$, and we have
  $\langle w , \psi \rangle_{L^2(\Omega\setminus\Omega_{c,-})} \leq 0$.
  Hence we conclude
  $\langle w , \psi \rangle_{L^2(\Omega\backslash \Omega_{c,\pm})} = 0$
  for all $\psi \in L^2(\Omega\setminus\Omega_{c,\pm})$
  satisfying
  \[
    \langle \psi , v \rangle_\Omega \leq
    \min \Big \{ \langle f_+ + \Delta u_\varrho , v \rangle_\Omega,
    - \langle f_- + \Delta u_\varrho , v \rangle_\Omega \Big \}
    \quad \mbox{for all} \; v \in H^1_0(\Omega), \; v \geq 0 ,
  \]
  and therefore $w(x)=0$ for almost all
  $x \in \Omega \backslash \Omega_{c,\pm}$ follows. Using
  Lemma \ref{Lemma regularization state constraints}, this implies
  \[
    \| \varrho \, \Delta u_\varrho \|_{L^2(\Omega\backslash \Omega_{c,\pm})} =
    \| u_\varrho - \overline{u} \|_{L^2(\Omega \backslash \Omega_{c,\pm})} \leq
    \| u_\varrho - \overline{u} \|_{L^2(\Omega)} \leq
    \varrho \, \| \Delta \overline{u} \|_{L^2(\Omega)},
  \]
  i.e.,
  \[
    \| \Delta u_\varrho \|_{L^2(\Omega \backslash \Omega_{c,\pm})} \leq
    \| \Delta \overline{u} \|_{L^2(\Omega)} .
  \]
  Together with $- \Delta u_\varrho = f_\pm \in L^2(\Omega_{c,\pm})$ this
  gives $- \Delta u_\varrho \in L^2(\Omega)$, and
  $- \Delta w \in L^2(\Omega)$.
  \end{proof}

\noindent
From the proof of Lemma \ref{Lemma 2.3} we already have
$f_-(x) < - \Delta u_\varrho(x) < f_+(x)$ and $w(x) = 0$ 
for $x \in \Omega \backslash \Omega_{c,\pm}$.
Moreover, \eqref{w psi +} gives
$\langle w , \psi \rangle_{L^2(\Omega_{c,-})} \geq 0$
for all $\psi \in L^2(\Omega_{c,-})$ with $\psi \geq 0$,
and hence we obtain
$- \Delta u_\varrho(x) = f_-(x)$ and $w(x) \geq 0$ for
$x \in \Omega_{c,-}$. In the same way we conclude
$- \Delta u_\varrho(x) = f_+(x)$ and $w(x) \leq 0$ for
$x \in \Omega_{c,+}$.
Note that these relations are the complimentarity conditions of the
variational inequality \eqref{VI control constraints u}.

\section{Finite element discretization}\label{Chapter FEM}
Let us consider the variational inequality to find $u_\varrho \in K$
such that
\begin{equation}\label{VI general}
  \varrho \, \langle \nabla u_\varrho , \nabla (v-u_\varrho)
  \rangle_{L^2(\Omega)} + \langle u_\varrho , v-u_\varrho \rangle_{L^2(\Omega)}
  \geq \langle \overline{u} , v - u_\varrho \rangle_{L^2(\Omega)}
  \quad \mbox{for all} \; v \in K,
\end{equation}
which corresponds to \eqref{VI state constraints} with $K=K_s$ in the
case of state constraints, and to \eqref{VI control constraints u} with
$K=K_c$ for control constraints. 
We now assume that $\Omega$ is either convex or sufficiently regular such
that $\| \Delta u \|_{L^2(\Omega)}$ defines an equivalent norm 
in $H^1_0(\Omega) \cap H^2(\Omega) = H^1_\Delta(\Omega)$. 

As in the unconstrained case, let
$V_h = S_h^1(\Omega) \cap H^1_0(\Omega) =
\mbox{span} \{ \varphi_k \}_{k=1}^M$
be a conforming finite element space, e.g., of piecewise linear and
continuous basis functions $\varphi_k$ which are defined with respect
to some admissible decomposition of $\Omega$ into simplicial shape
regular finite elements $\tau_\ell$ of local mesh size $h_\ell$. 

Let $K_h \subset V_h$ be some appropriate approximation of $K$ to
be specified later. Then we consider the Galerkin variational inequality of
\eqref{VI general} to find $u_{\varrho h} \in K_h$ such that
\begin{equation}\label{VI general FEM}
  \varrho \, \langle \nabla u_{\varrho h} ,
  \nabla (v_h - u_{\varrho h}) \rangle_{L^2(\Omega)} +
  \langle u_{\varrho h} , v_h - u_{\varrho h} \rangle_{L^2(\Omega)} \geq
  \langle \overline{u} , v_h - u_{\varrho h} \rangle_{L^2(\Omega)}
\end{equation}
is satisfied for all $v_h \in K_h$, which is obviously equivalent to
\begin{equation}\label{VI general FEM orthogonality}
  \langle \overline{u} - u_{\varrho h} ,
  v_h - u_{\varrho h} \rangle_{L^2(\Omega)} - \varrho \,
  \langle \nabla u_{\varrho h} , \nabla (v_h - u_{\varrho h})
  \rangle_{L^2(\Omega)}
  \leq 0 \quad \mbox{for all} \; v_h \in K_h .
\end{equation}
Following \cite{Falk:1974} we can prove the following
a priori error estimate for the solution $u_{\varrho h} \in K_h$ of
the variational inequality \eqref{VI general FEM orthogonality}.

\begin{lemma}
  For $ u_\varrho \in K$ and $u_{\varrho h} \in K_h$ being the unique
  solutions of the variational inequalities \eqref{VI general}
  and \eqref{VI general FEM}, respectively, there holds the error estimate
  \begin{eqnarray}\label{VI general error}
    && \| u_\varrho - u_{\varrho h} \|^2_{L^2(\Omega)} +
    \varrho \,
    \| \nabla (u_\varrho - u_{\varrho h}) \|_{L^2(\Omega)}^2 \\
    && \hspace*{1mm} \leq \, 8 \, \Big[ \inf\limits_{v_h \in K_h} \Big(
     \| u_\varrho - v_h \|_{L^2(\Omega)}^2 + \varrho \,
     \| \nabla (u_\varrho - v_h) \|_{L^2(\Omega)}^2 \Big) + \,
       \varrho^2 \, \| \Delta u_\varrho \|^2_{L^2(\Omega)} +
       \| u_\varrho - \overline{u} \|_{L^2(\Omega)}^2 \Big] .\nonumber
  \end{eqnarray}
\end{lemma}

\begin{proof}
  For arbitrary $v_h \in K_h$,
  using \eqref{VI general FEM orthogonality} and integration by parts,
  we can write
  \begin{eqnarray*}
    && \| u_\varrho - u_{\varrho h} \|^2_{L^2(\Omega)} +
     \varrho \, \| \nabla (u_\varrho - u_{\varrho h}) \|_{L^2(\Omega)}^2 \\
    && \hspace*{5mm} = \,
       \langle u_\varrho - u_{\varrho h}, u_\varrho - u_{\varrho h}
       \rangle_{L^2(\Omega)} + \varrho \,
       \langle \nabla (u_\varrho - u_{\varrho h}), \nabla
       (u_\varrho - u_{\varrho h}) \rangle_{L^2(\Omega)} \\
  && \hspace*{5mm} = \,
       \langle u_\varrho - u_{\varrho h}, u_\varrho - v_h
       \rangle_{L^2(\Omega)} + \varrho \,
       \langle \nabla (u_\varrho - u_{\varrho h}), \nabla
     (u_\varrho - v_h) \rangle_{L^2(\Omega)} \\
    && \hspace*{10mm} + 
       \langle u_\varrho - u_{\varrho h}, v_h - u_{\varrho h}
       \rangle_{L^2(\Omega)} + \varrho \,
       \langle \nabla (u_\varrho - u_{\varrho h}), \nabla
       (v_h - u_{\varrho h}) \rangle_{L^2(\Omega)} \\    
  && \hspace*{5mm} = \,
       \langle u_\varrho - u_{\varrho h}, u_\varrho - v_h
       \rangle_{L^2(\Omega)} + \varrho \,
       \langle \nabla (u_\varrho - u_{\varrho h}), \nabla
     (u_\varrho - v_h) \rangle_{L^2(\Omega)} \\
    && \hspace*{10mm} + 
       \langle u_\varrho - \overline{u}, v_h - u_{\varrho h}
       \rangle_{L^2(\Omega)} + \varrho \,
       \langle \nabla u_\varrho , \nabla
       (v_h - u_{\varrho h}) \rangle_{L^2(\Omega)} \\
    && \hspace*{15mm}
       + \langle \overline{u} - u_{\varrho h}, v_h - u_{\varrho h}
       \rangle_{L^2(\Omega)}
       - \varrho \,
       \langle \nabla u_{\varrho h} , \nabla
       (v_h - u_{\varrho h}) \rangle_{L^2(\Omega)} \\
  && \hspace*{5mm} \leq \,
       \langle u_\varrho - u_{\varrho h}, u_\varrho - v_h
       \rangle_{L^2(\Omega)} + \varrho \,
       \langle \nabla (u_\varrho - u_{\varrho h}), \nabla
     (u_\varrho - v_h) \rangle_{L^2(\Omega)} \\
    && \hspace*{10mm} + 
       \langle u_\varrho - \overline{u}, v_h - u_{\varrho h}
       \rangle_{L^2(\Omega)} + \varrho \,
       \langle \nabla u_\varrho , \nabla
       (v_h - u_{\varrho h}) \rangle_{L^2(\Omega)} \\
  && \hspace*{5mm} = \,
       \langle u_\varrho - u_{\varrho h}, u_\varrho - v_h
       \rangle_{L^2(\Omega)} + \varrho \,
       \langle \nabla (u_\varrho - u_{\varrho h}), \nabla
     (u_\varrho - v_h) \rangle_{L^2(\Omega)} \\
    && \hspace*{10mm} + 
       \langle - \varrho \Delta u_\varrho
       + u_\varrho - \overline{u}, v_h - u_{\varrho h}
       \rangle_{L^2(\Omega)} \\
  && \hspace*{5mm} \leq \,
     \| u_\varrho - u_{\varrho h} \|_{L^2(\Omega)}
     \| u_\varrho - v_h \|_{L^2(\Omega)} + \varrho \,
     \| \nabla (u_\varrho - u_{\varrho h}) \|_{L^2(\Omega)}
     \| \nabla (u_\varrho - v_h) \|_{L^2(\Omega)}  \\
    && \hspace*{10mm} + 
       \| - \varrho \Delta u_\varrho
       + u_\varrho - \overline{u} \|_{L^2(\Omega)}
       \| v_h - u_{\varrho h} \|_{L^2(\Omega)} \, .
  \end{eqnarray*}
  When using Young's inequality we further have
  \begin{eqnarray*}
    \| u_\varrho - u_{\varrho h} \|^2_{L^2(\Omega)} +
     \varrho \, \| \nabla (u_\varrho - u_{\varrho h}) \|_{L^2(\Omega)}^2 
    & \leq & \frac{1}{4} \, \| u_\varrho - u_{\varrho h} \|_{L^2(\Omega)}^2 +
             \| u_\varrho - v_h \|_{L^2(\Omega)}^2 \\
    && \hspace*{-2cm} 
     + \frac{1}{2} \varrho \,
     \| \nabla (u_\varrho - u_{\varrho h}) \|_{L^2(\Omega)}^2
     + \frac{1}{2} \varrho
     \| \nabla (u_\varrho - v_h) \|_{L^2(\Omega)}^2  \\
    && \hspace*{-2cm} + 
       \| - \varrho \Delta u_\varrho
       + u_\varrho - \overline{u} \|^2_{L^2(\Omega)} +
       \frac{1}{4} \, \| v_h - u_{\varrho h} \|_{L^2(\Omega)}^2 \, ,
  \end{eqnarray*}
  i.e.,
  \begin{eqnarray*}
    && \hspace*{-5mm}
       \frac{3}{4} \, \| u_\varrho - u_{\varrho h} \|^2_{L^2(\Omega)} +
    \frac{1}{2} \, \varrho \,
    \| \nabla (u_\varrho - u_{\varrho h}) \|_{L^2(\Omega)}^2 \, \leq \,
     \| u_\varrho - v_h \|_{L^2(\Omega)}^2 + \frac{1}{2} \, \varrho \,
     \| \nabla (u_\varrho - v_h) \|_{L^2(\Omega)}^2  \\
    && \hspace*{10mm} + \,
       \Big( \varrho \, \| \Delta u_\varrho \|_{L^2(\Omega)} +
       \| u_\varrho - \overline{u} \|_{L^2(\Omega)} \Big)^2  +
       \frac{1}{4} \, \Big(
       \| v_h - u_\varrho \|_{L^2(\Omega)} +
       \| u_\varrho - u_{\varrho h} \|_{L^2(\Omega)} \Big)^2 \\
    && \hspace*{5mm} \leq \,
     \| u_\varrho - v_h \|^2_{L^2(\Omega)} + \frac{1}{2} \, \varrho \,
     \| \nabla (u_\varrho - v_h) \|_{L^2(\Omega)}^2  \\
    && \hspace*{10mm} + \,
       2 \, \varrho^2 \, \| \Delta u_\varrho \|^2_{L^2(\Omega)} +
       2 \, \| u_\varrho - \overline{u} \|_{L^2(\Omega)}^2 +
       \frac{1}{2} \, 
       \| v_h - u_\varrho \|^2_{L^2(\Omega)} +
       \frac{1}{2} \, \| u_\varrho - u_{\varrho h} \|_{L^2(\Omega)}^2,
  \end{eqnarray*}
  and hence,
  \begin{eqnarray*}
    && \frac{1}{4} \, \| u_\varrho - u_{\varrho h} \|^2_{L^2(\Omega)} +
    \frac{1}{2} \, \varrho \,
    \| \nabla (u_\varrho - u_{\varrho h}) \|_{L^2(\Omega)}^2 \\
    && \hspace*{5mm} \leq \, \frac{3}{2} \,
     \| u_\varrho - v_h \|_{L^2(\Omega)}^2 + \frac{1}{2} \, \varrho \,
     \| \nabla (u_\varrho - v_h) \|_{L^2(\Omega)}^2 + \,
       2 \, \varrho^2 \, \| \Delta u_\varrho \|^2_{L^2(\Omega)} +
       2 \, \| u_\varrho - \overline{u} \|_{L^2(\Omega)}^2 ,
  \end{eqnarray*}
  and the assumption follows.
  \end{proof}

\subsection{State constraints}
Let $I_h : C(\overline{\Omega}) \to S_h^1(\Omega)$ be the nodal
interpolation operator. When assuming $g_\pm \in H^1_{\Delta}(\Omega)=H^1_0(\Omega)\cap H^2(\Omega)$ we
then define
\[
  K_{sh} := \Big \{ v_h \in V_h : I_h g_- \leq v_h \leq I_h g_+ \;
  \mbox{in} \; \Omega  \Big \},
\]
and we consider the variational inequality
\eqref{VI general FEM} for $K_h = K_{sh}$.

\begin{theorem}\label{Thm error state}
  Let $u_\varrho \in K_s$ and $u_{\varrho h} \in K_{sh}$ be the unique
  solutions of the variational inequalities \eqref{VI state constraints}
  and \eqref{VI general FEM}, respectively. Assume
  $\overline{u} \in K_s\cap H^1_{\Delta}(\Omega)$ and $g_\pm \in H^1_{\Delta}(\Omega)$.
  When choosing $\varrho=h^2$, then there holds the error estimate
\begin{equation}\label{VI State FEM error}
  \| u_\varrho - u_{\varrho h} \|^2_{L^2(\Omega)} +
  h^2 \, \| \nabla (u_\varrho - u_{\varrho h}) \|_{L^2(\Omega)}^2 \leq
  c \, h^4 \, \Big[ \| \Delta \overline{u} \|^2_{L^2(\Omega)} +
  \|\Delta g_\pm\|^2_{L^2(\Omega)} \Big] .
\end{equation}
\end{theorem}

\begin{proof}
  Due to Lemma \ref{Lemma 2.2} we have
  \[
    \| \Delta u_\varrho \|^2_{L^2(\Omega)} =
    \| \Delta g_\pm \|^2_{L^2(\Omega_\pm)} +
    \| \Delta u_\varrho \|^2_{L^2(\Omega \backslash \Omega_\pm)}, 
  \]
  and since $\| \Delta u \|_{L^2(\Omega)}$ defines an equivalent
  norm in $H^1_0(\Omega) \cap H^2(\Omega)$,
  \[
    |u_\varrho|^2_{H^2(\Omega)} \leq c \,
    \| \Delta u_\varrho \|^2_{L^2(\Omega)}
    \leq c \, \Big[
    \| \Delta \overline{u} \|_{L^2(\Omega)} +  \|\Delta g_\pm\|^2_{L^2(\Omega)} \Big]
  \]
  follows. Hence we can consider
  the nodal interpolation $I_h u_\varrho \in K_{sh}$ and we can use
  standard interpolation error estimates to conclude
  \[
     \| u_\varrho - I_hu_\varrho \|^2_{L^2(\Omega)} + \varrho \,
     \| \nabla (u_\varrho - I_h u_\varrho) \|_{L^2(\Omega)}^2 \leq
     c \, \Big( h^4 + \varrho \, h^2 \Big) \,
     |u_\varrho|^2_{H^2(\Omega)} .
   \]
   With this and using \eqref{regularization state constraints H Su}
   we can write the general error estimate \eqref{VI general error} as
  \begin{eqnarray*}
    && \| u_\varrho - u_{\varrho h} \|^2_{L^2(\Omega)} +
    \varrho \, \| \nabla (u_\varrho - u_{\varrho h}) \|_{L^2(\Omega)}^2 \\
    && \hspace*{2cm} \leq \, 8 \, \Big[
       c \, \Big( h^4 + \varrho \, h^2 \Big) \,
       |u_\varrho|^2_{H^2(\Omega)}
       + 2 \, \varrho^2 \, \Big( \| \Delta \overline{u} \|^2_{L^2(\Omega)}
       + \|\Delta g_\pm\|^2_{L^2(\Omega)} \Big)  
       \Big] .
     \end{eqnarray*}
     The assertion finally follows when choosing $\varrho = h^2$.
     \end{proof}

 \begin{corollary}
   Assume $\overline{u} \in K_s \cap H^r(\Omega)$ for some $r \in (1,2]$,
   or $\overline{u} \in H^r_0(\Omega)$ for some $r \in [0,1]$.
   In the latter case we also assume
   $g_-(x) \leq \overline{u}(x) \leq g_+(x)$ for almost
   all $x \in \Omega$, where $g_\pm \in H^r(\Omega)$, $r \in [0,2]$.
   Then there holds the error estimate
   \begin{equation}\label{Error state L2 Hs}
     \| u_{\varrho h} - \overline{u} \|_{L^2(\Omega)} \leq c \, h^r \,
     \sqrt{ \| \overline{u} \|^2_{H^r(\Omega)} + \| g_\pm\|^2_{H^r(\Omega)}}.
   \end{equation}
\end{corollary}

\begin{proof}
  When considering the variational inequality
  \eqref{VI general FEM} for $v_h =0$, this gives
  \[
    \varrho \, \langle \nabla u_{\varrho h} ,
    \nabla u_{\varrho h} \rangle_{L^2(\Omega)} +
    \langle u_{\varrho h} - \overline{u} ,
    u_{\varrho h} - \overline{u} \rangle_{L^2(\Omega)} \leq
    \langle \overline{u} - u_{\varrho h} ,
    \overline{u} \rangle_{L^2(\Omega)} ,
  \]
  from which we derive the trivial error estimate
  \[
    \| u_{\varrho h} - \overline{u} \|_{L^2(\Omega)} \leq
    \| \overline{u} \|_{L^2(\Omega)} \leq
    \sqrt{\| \overline{u} \|^2_{L^2(\Omega)} +
      \| g_\pm \|^2_{L^2(\Omega)}} .
  \]
  On the other hand, \eqref{VI State FEM error} and
  \eqref{regularization state constraints H Su} imply, recall
  $\varrho = h^2$,
  \[
    \| u_{\varrho h} - \overline{u} \|_{L^2(\Omega)} \leq
    \| u_{\varrho h} - u_\varrho \|_{L^2(\Omega)} +
    \| u_\varrho - \overline{u} \|_{L^2(\Omega)} \leq
    c \, h^2 \, \sqrt{\| \overline{u} \|^2_{H^2(\Omega)} +
      \| g_\pm \|^2_{H^2(\Omega)}}.
  \]
  Now the assertion follows from a space interpolation argument.
  \end{proof}
   
\noindent
Using the isomorphism $v_h \in K_{sh} \leftrightarrow
\underline{v} \in {\mathbb{R}}^M$ we can write the variational
inequality \eqref{VI general FEM} as: Find $\underline{u} \in {\mathbb{R}}^M
\leftrightarrow u_{\varrho h} \in K_{sh}$ such that
\begin{equation}\label{state LGS}
  \varrho \, ( K_h \underline{u} , \underline{v} - \underline{u} ) +
  ( M_h \underline{u} , \underline{v} - \underline{u} ) \geq
  ( \underline{\overline{u}} , \underline{v} - \underline{u} )
\end{equation}
is satisfied for all $\underline{v} \in {\mathbb{R}}^M
\leftrightarrow v_h \in K_{sh}$. Here, $M_h$ and $K_h$ are the
standard finite element mass and stiffness matrices, defined by
\[
  K_h[j,k] = \int_\Omega \nabla \varphi_k(x) \cdot \nabla \varphi_j(x) \, dx,
  \quad
  M_h[j,k] = \int_\Omega \varphi_k(x) \varphi_j(x) \, dx, \quad
  j,k = 1,\ldots,M,
\]
and $\underline{\overline{u}}$ is the load vector with the entries
\[
  \overline{u}_j = \int_\Omega \overline{u}(x) \, \varphi_j(x) \, dx \quad
  \mbox{for} \; j=1,\ldots,M.
\]
As in the continuous case we define
$\underline{\lambda} := M_h \underline{u} + \varrho K_h \underline{u} -
\underline{\overline{u}} \in {\mathbb{R}}^M$.
Further, let the index set of the active nodes be denoted by 
$D_\pm := \{ k=1,\ldots,M : u_k = g_{\pm,k} \}$.
Then we conclude the discrete complementarity conditions
\begin{equation}\label{VI state discrete complementarity}
  \lambda_k = 0, \; g_{-,k} < u_k < g_{+,k} \; \mbox{for} \;
  k \not\in D_\pm, \; \lambda_k \leq 0 \; \mbox{for} \; k \in D_+, \;
  \lambda_k \geq 0 \; \mbox{for} \; k \in D_- ,
\end{equation}
which are equivalent to
\[
  \lambda_k = \min \{ 0, \lambda_k + c (g_{+,k}-u_k) \} +
  \max \{ 0, \lambda_k + c (g_{-,k}-u_k) \} , \quad c > 0.
\]
Hence we have to solve a system
$\underline{F}(\underline{u},\underline{\lambda})=\underline{0}$
of (non)linear equations
\begin{equation}\label{nonlinear 1 state}
  F_1(\underline{u},\underline{\lambda}) =
  M_h \underline{u} + \varrho K_h \underline{u} - \underline{\lambda} -
  \underline{\overline{u}} = \underline{0},
\end{equation}
\begin{equation}\label{nonlinear 2 state}
  F_2(\underline{u} , \underline{\lambda}) =
  \underline{\lambda} -
  \min \{ 0 , \underline{\lambda} + c (\underline{g}_+ - \underline{u}) \} -
  \max \{ 0 , \underline{\lambda} + c (\underline{g}_--\underline{u}) \},
\end{equation}
where the latter have to be considered componentwise. This
will be discussed in Section~\ref{Section:Newton}.

\subsection{Control constraints}
In the case of control constraints we consider the variational
inequality \eqref{VI general FEM} for $K_h=K_{ch}$ where
\[
  K_{ch} := \Big \{ u_h \in V_h : \langle f_- , v_h \rangle_{L^2(\Omega)}
  \leq \langle \nabla u_h , \nabla v_h \rangle_{L^2(\Omega)} \leq
  \langle f_+ , v_h \rangle_{L^2(\Omega)} \; \forall v_h \in V_h , \;
  v_h \geq 0 \Big \} \, .
\]

\begin{theorem}\label{Thm error control}
  Let $u_\varrho \in K_c$ and $u_{\varrho h} \in K_{ch}$ be the unique
  solutions of the variational inequalities \eqref{VI control constraints u}
  and \eqref{VI general FEM}, respectively. Assume
  $\overline{u} \in K_c\cap H^1_{\Delta}(\Omega)$ and $f_\pm \in L^2(\Omega)$.
  When choosing $\varrho=h^2$, then there holds the error estimate
\begin{equation}\label{VI Control FEM error}
  \| u_\varrho - u_{\varrho h} \|^2_{L^2(\Omega)} +
  h^2 \, \| \nabla (u_\varrho - u_{\varrho h}) \|_{L^2(\Omega)}^2 \leq
  c \, h^4 \, \Big[ \| \Delta \overline{u} \|^2_{L^2(\Omega)} +
  \|f_\pm\|^2_{L^2(\Omega)} \Big] .
\end{equation}
\end{theorem}

\begin{proof}
  Due to Lemma \ref{Lemma 2.3} we have
  \[
    \| \Delta u_\varrho \|^2_{L^2(\Omega)} =
    \| f_\pm \|^2_{L^2(\Omega_\pm)} +
    \| \Delta u_\varrho \|^2_{L^2(\Omega \backslash \Omega_\pm)} \leq
    \| \Delta \overline{u} \|_{L^2(\Omega)}^2 + \|f_\pm\|^2_{L^2(\Omega)} ,
  \]
  and
  \[
    |u_\varrho|^2_{H^2(\Omega)} \leq c \,
    \| \Delta u_\varrho \|^2_{L^2(\Omega)}
    \leq c \, \Big[
    \| \Delta \overline{u} \|_{L^2(\Omega)}^2 + \|f_\pm\|^2_{L^2(\Omega)} \Big]
  \]
  follows. For $u_\varrho \in K_c$ we define $P_h u_\varrho \in V_h$
  as the unique solution of the variational formulation
  \[
    \langle \nabla P_h u_\varrho , \nabla v_h \rangle_{L^2(\Omega)}
    =
    \langle \nabla u_\varrho , \nabla v_h \rangle_{L^2(\Omega)}
    \quad \mbox{for all} \; v_h \in V_h.
  \]
  By construction we have $P_h u_\varrho \in K_{ch}$, and using standard
  finite element error estimates including the Nitsche trick we conclude
  \[
    \| u_\varrho - P_h u_\varrho \|^2_{L^2(\Omega)} +
    \varrho \, \| \nabla (u_\varrho - P_h u_\varrho) \|^2_{L^2(\Omega)}
    \leq c \, \Big( h^4 + \varrho \, h^2 \Big) \,
    |u_\varrho|^2_{H^2(\Omega)} .
  \]
  The assertion now follows as in the case of state constraints, we
  skip the details.
  \end{proof}

\noindent
As in the case of state constraints we can also derive error estimates
for less regular target functions.

Using the isomorphism $v_h \in K_{ch} \leftrightarrow
\underline{v} \in {\mathbb{R}}^M$ we can write the variational
inequality \eqref{VI general FEM}
as: Find $\underline{u} \in {\mathbb{R}}^M \leftrightarrow
u_{\varrho h} \in K_{ch}$ such that
\[
  ( (M_h + \varrho K_h ) \underline{u} - \underline{\overline{u}},
  \underline{v} - \underline{u} ) \geq 0 \quad
  \mbox{for all} \; \underline{v} \in {\mathbb{R}}^M \leftrightarrow
  v_h \in K_{ch} \, .
\]
The control constraints $v_h \in K_{ch}$ are equivalent to
\[
  f_{-,i} \leq (K_h \underline{v})_i \leq f_{+,i} \quad
  \mbox{for all} \; i=1,\ldots,M \, .
\]
On the other hand, the discrete variational inequality can be written as
\[
  ( \underline{w} ,   K_h \underline{v} - K_h \underline{u} ) \geq 0 ,
  \quad \mbox{where} \quad
\underline{w} :=
(K_h^{-1} M_h + \varrho I ) \underline{u} - K_h^{-1}
  \underline{\overline{u}} .
\]
We introduce the discrete active sets
$I_\pm := \{ i \in {\mathbb{R}}^M : (K_h \underline{u})_i = f_{\pm,i} \}$
and conclude
\[
  \sum\limits_{i \in I_+} w_i
  \underbrace{[(K_h \underline{v})_i - f_{+,i}]}_{\leq 0}
+
\sum\limits_{i \in I_-} w_i
\underbrace{[(K_h \underline{v})_i - f_{-,i}]}_{\geq 0}
+
\sum\limits_{i \in I \backslash I_\pm} w_i
[(K_h\underline{v})_i - (K_h \underline{u})_i] \geq 0 \, .
\]
Let us define
$g_i = f_{+,i}$ for $i \in I_+$, 
$g_i = f_{-,i}$ for $i \in I_-$, and
$g_i = (K_h \underline{u})_i$ for $i \in I \backslash I_\pm$.
For some $j \in I_+$ we set
$g_j = f_{+,j} - \alpha$ with $0 < \alpha < f_{+,j} - f_{-,j}$, and
we solve $K_h \underline{v} = \underline{g}$.
By construction we obtain
$- w_j  \alpha \geq 0$, i.e., $w_j \leq 0$ for $j \in I_+$,
and in the same way we conclude $w_j \geq 0$ for $j \in I_-$ as
well as $w_j=0$ for $j \in I \backslash I_\pm$.
Hence we have the discrete complementarity conditions
\[
  w_j = 0 \; : \; f_{-,j} < (K_h \underline{u})_j < f_{+,j},
  \, j \not\in I_\pm, \;
  w_j \leq 0 \; : \; (K_h\underline{u})_j = f_{+,j},\;
    w_j \geq 0 \; : \; (K_h\underline{u})_j = f_{-,j}.
\]
As in the case of state constraints we have to solve a system of
(non)linear equations,
\begin{equation}\label{nonlinear 1 control}
F_1(\underline{u},\underline{w}) = 
M_h\underline{u} + \varrho K_h  \underline{u} - K_h \underline{w}
- \underline{\overline{u}} = \underline{0},
\end{equation}
\begin{equation}\label{nonlinear 2 control}
  F_2(\underline{u} , \underline{w}) =
  \underline{w} -
  \min \{ 0 , \underline{w} + c (\underline{f}_+ - K_h \underline{u}) \} -
  \max \{ 0 , \underline{w} + c (\underline{f}_-- K_h \underline{u}) \},
  \quad c > 0 .
\end{equation}

\subsection{Finite element approximation of the control}
When the state $u_\varrho$, i.e., its finite element approximation
$u_{\varrho h}$, is known it remains to find the related control
$z_\varrho = -\Delta u_\varrho  \in H^{-1}(\Omega)$, i.e., an appropriate
finite element approximation of
$\widetilde{z}_\varrho = - \Delta u_{\varrho h} \in H^{-1}(\Omega)$.
To do so,
we define $A : H^1_0(\Omega) \to H^{-1}(\Omega)$ satisfying
\[
  \langle A u , v \rangle_\Omega =
  \langle \nabla u , \nabla v \rangle_{L^2(\Omega)} \quad
  \mbox{for all} \; u,v \in H^1_0(\Omega) ,
\]
and we can compute $\widetilde{z}_\varrho \in H^{-1}(\Omega)$ as
unique solution of the variational formulation
\[
  \langle A^{-1} \widetilde{z}_\varrho , \psi \rangle_\Omega =
  \langle u_{\varrho h} , \psi \rangle_{L^2(\Omega)} \quad
  \mbox{for all} \; \psi \in H^{-1}(\Omega).
\]
In addition to the finite element space $V_h \subset H^1_0(\Omega)$ of
piecewise linear and continuous basis functions $\varphi_k$ we now define the
ansatz space $Z_H = S_H^0(\Omega) = \mbox{span} \{ \psi_\ell \}_{\ell=1}^N$
of piecewise constant basis functions $\psi_\ell$ which are defined with
respect to some mesh of mesh size $H \simeq h$.
Then we can find the finite element
approximation $\widetilde{z}_{\varrho H} \in Z_H$ as unique solution of the
variational formulation
\begin{equation}\label{FEM control}
  \langle A^{-1} \widetilde{z}_{\varrho H} , \psi_H \rangle_\Omega =
  \langle u_{\varrho h} , \psi_H \rangle_{L^2(\Omega)} \quad
  \mbox{for all} \; \psi_H \in Z_H .
\end{equation}
In addition, let $z_{\varrho H} \in Z_H$ be the solution of the variational
formulation
\[
  \langle A^{-1} z_{\varrho H} , \psi_H \rangle_\Omega =
  \langle u_\varrho  , \psi_H \rangle_{L^2(\Omega)} =
  \langle A^{-1} z_\varrho , \psi_H \rangle_\Omega \quad
  \mbox{for all} \; \psi_H \in Z_H .
\]
When using standard arguments we conclude Cea's lemma
\[
  \| z_\varrho - z_{\varrho H} \|_{H^{-1}(\Omega)} \leq
  \inf\limits_{\psi_H \in Z_H} \| z_\varrho - \psi_H \|_{H^{-1}(\Omega)},
\]
and the error estimate
\[
  \| z_\varrho - z_{\varrho H} \|_{H^{-1}(\Omega)}
  \leq c \, H \, \| z_\varrho \|_{L^2(\Omega)} =
  c \, H \, \| \Delta u_\varrho \|_{L^2(\Omega)}.
\]
In the case of state constraints we further have,
as in the proof of Lemma \ref{Lemma 2.2},
\begin{eqnarray*}
  \| \Delta u_\varrho \|_{L^2(\Omega)}^2
  & = & \| \Delta g_\pm \|_{L^2(\Omega_\pm)}^2 +
        \| \Delta u_\varrho \|_{L^2(\Omega \backslash \Omega_\pm)}^2 \, = \,
        \| \Delta g_\pm \|_{L^2(\Omega_\pm)}^2 + \frac{1}{\varrho^2} \,
        \| \overline{u} - u_\varrho \|^2_{L^2(\Omega\backslash \Omega_\pm)} \\
  & \leq & \| \Delta g_\pm \|_{L^2(\Omega)}^2 + \varrho^{s-2} \,
           \| \overline{u} \|^2_{H^s(\Omega)} \, \leq \,
           \Big( \| \Delta g_\pm \|_{L^2(\Omega)} + \varrho^{s/2-1} \,
           \| \overline{u} \|_{H^s(\Omega)} \Big)^2,
\end{eqnarray*}
i.e., recall $H \simeq h$ and $\varrho = h^2$,
\[
  \| z_\varrho - z_{\varrho H} \|_{H^{-1}(\Omega)}
  \leq c \, h^{s-1} \,
  \Big( \| g_\pm \|_{H^2(\Omega)} + 
  \| \overline{u} \|_{H^s(\Omega)} \Big) .
\]
Note that in the case of control constraints we obtain a similar result,
when assuming $f_\pm \in L^2(\Omega)$ instead of
$g_\pm \in H^1_{\Delta}(\Omega)$.
In any case, we have the perturbed Galerkin orthogonality
\[
  \langle A^{-1} (z_{\varrho H} - \widetilde{z}_{\varrho H}), \psi_H
  \rangle_\Omega = \langle u_\varrho - u_{\varrho h} , \psi_H
  \rangle_{L^2(\Omega)}
  \quad \mbox{for all} \; \psi_H \in Z_H,
\]
from which we conclude
\begin{eqnarray*}
  && \| z_{\varrho H} - \widetilde{z}_{\varrho H} \|^2_{H^{-1}(\Omega)}
  \, = \,\langle A^{-1} (z_{\varrho H} - \widetilde{z}_{\varrho H}),
        z_{\varrho H} - \widetilde{z}_{\varrho H} \rangle_\Omega \, = \,
        \langle u_\varrho - u_{\varrho h} ,
        z_{\varrho H} - \widetilde{z}_{\varrho H} \rangle_{L^2(\Omega)} \\
  && \leq \, \| u_\varrho - u_{\varrho h} \|_{L^2(\Omega)}
           \| z_{\varrho H} - \widetilde{z}_{\varrho H} \|_{L^2(\Omega)} 
  \, \leq \, c \, h^s \, \sqrt{
           \| \overline{u} \|^2_{H^s(\Omega)} + \| g_{\pm} \|^2_{H^s(\Omega)}} \, 
           H^{-1} \, \| z_{\varrho H} - \widetilde{z}_{\varrho H}
           \|_{H^{-1}(\Omega)},
\end{eqnarray*}
when using an inverse inequality in $Z_H$, and the related error
estimates for the approximate state. This gives
\[
  \| z_{\varrho H} - \widetilde{z}_{\varrho H} \|_{H^{-1}(\Omega)} \leq
  c \, h^{s-1} \, \Big(
  \| \overline{u} \|^2_{H^s(\Omega)} + \| g_{\pm} \|^2_{H^2(\Omega)}
  \Big),
\]
and, therefore,
\begin{eqnarray*}
  \| z_\varrho - \widetilde{z}_{\varrho H} \|_{H^{-1}(\Omega)}
  & \leq & \| z_\varrho - z_{\varrho H} \|_{H^{-1}(\Omega)} +
           \| z_{\varrho H} - \widetilde{z}_{\varrho H} \|_{H^{-1}(\Omega)} 
  \, \leq \, c \, h^{s-1} \, \Big(
  \| \overline{u} \|^2_{H^s(\Omega)} + \| g_{\pm} \|^2_{H^2(\Omega)}
  \Big)
\end{eqnarray*}
follows. Note that we cannot expect any order of convergence for the
approximate control in $H^{-1}(\Omega)$ when we have
$\overline{u} \in H^s(\Omega)$ for $s < 1$ only. In this case we have
to measure the error in a weaker norm. For this we consider,
using the $L^2$ projection $Q_H : L^2(\Omega) \to Z_H$, 
\begin{eqnarray*}
  \| z_\varrho - \widetilde{z}_{\varrho H} \|_{H^{-2}(\Omega)}
  & = & \sup\limits_{0 \neq v \in H^1_0(\Omega) \cap H^2(\Omega)}
        \frac{\langle z_\varrho - \widetilde{z}_{\varrho H}, v \rangle_\Omega}
        {\| v \|_{H^2(\Omega)}} \\
  && \hspace*{-3.7cm} = \,
     \sup\limits_{0 \neq v = A^{-1}\psi \in H^1_0(\Omega) \cap H^2(\Omega)}
     \frac{\langle z_\varrho - \widetilde{z}_{\varrho H},
     A^{-1} \psi \rangle_\Omega}{\| v \|_{H^2(\Omega)}} \\
  && \hspace*{-3.7cm}
     = \sup\limits_{0 \neq v = A^{-1}\psi \in H^1_0(\Omega) \cap H^2(\Omega)}
     \frac{\langle A^{-1} (z_\varrho - \widetilde{z}_{\varrho H}),
     \psi \rangle_\Omega}{\| v \|_{H^2(\Omega)}} \\
  && \hspace*{-3.7cm}
     = \sup\limits_{0 \neq v = A^{-1}\psi \in H^1_0(\Omega) \cap H^2(\Omega)}
     \frac{\langle A^{-1} (z_\varrho - \widetilde{z}_{\varrho H}),
     \psi - Q_H \psi \rangle_\Omega +
     \langle A^{-1} (z_\varrho - \widetilde{z}_{\varrho H}),
     Q_H \psi \rangle_\Omega }{\| v \|_{H^2(\Omega)}} \\
  && \hspace*{-3.7cm}
     = \sup\limits_{0 \neq v = A^{-1}\psi \in H^1_0(\Omega) \cap H^2(\Omega)}
     \frac{\langle A^{-1} (z_\varrho - \widetilde{z}_{\varrho H}),
     \psi - Q_H \psi \rangle_\Omega + \langle u_\varrho - u_{\varrho h},
     Q_H \psi \rangle_\Omega }{\| v \|_{H^2(\Omega)}} \\
  && \hspace*{-3.7cm}
     = \sup\limits_{0 \neq v = A^{-1}\psi \in H^1_0(\Omega) \cap H^2(\Omega)}
     \frac{\| z_\varrho - \widetilde{z}_{\varrho H} \|_{H^{-1}(\Omega)}
     \| \psi - Q_H \psi \|_{H^{-1}(\Omega)} +
     \| u_\varrho - u_{\varrho h} \|_{L^2(\Omega)}
     \| Q_H \psi \|_{L^2(\Omega)} }{\| v \|_{H^2(\Omega)}} \\
  && \hspace*{-3.7cm}
     \leq \sup\limits_{0 \neq v = A^{-1}\psi \in H^1_0(\Omega) \cap H^2(\Omega)}
     \frac{c \, H \,
     \| z_\varrho - \widetilde{z}_{\varrho H} \|_{H^{-1}(\Omega)}
     \| \psi \|_{L^2(\Omega)} + \| u_\varrho - u_{\varrho h} \|_{L^2(\Omega)}
     \| \psi \|_{L^2(\Omega)} }{\| v \|_{H^2(\Omega)}} \\[2mm]
  && \hspace*{-3.7cm} \leq \, c \, H \,
     \| z_\varrho - \widetilde{z}_{\varrho H} \|_{H^{-1}(\Omega)}
     + \| u_\varrho - u_{\varrho h} \|_{L^2(\Omega)} 
     \, = \, c \, h^s \, \Big(
     \| \overline{u} \|_{H^s(\Omega)} +
     \| g_\pm \|_{H^2(\Omega)} \Big) ,
\end{eqnarray*}
when assuming $\overline{u} \in H^s_0(\Omega)$ for $s \in [0,1]$, or
$\overline{u} \in H^1_0(\Omega) \cap H^s(\Omega)$ for $s \in (1,2]$.
In particular, this error estimate also allows the consideration of
discontinuous target functions $\overline{u} \in H^s_0(\Omega)$,
$s < 1/2$. However, the application of the inverse $A^{-1}$ does not
allow a direct evaluation, and hence we need to introduce a suitable
approximation as follows:
For any $z \in H^{-1}(\Omega)$ we define $p_z \in H^1_0(\Omega)$
as the unique solution of the variational formulation
\[
  \langle A p_z , q \rangle_\Omega =
  \langle \nabla p_z , \nabla q \rangle_{L^2(\Omega)} =
  \langle z , q \rangle_\Omega \quad \mbox{for all} \; q \in H^1_0(\Omega).
\]
Moreover, we compute an approximate solution
$p_{z h} \in V_h$, for simplicity we consider the finite element space
$V_h$ as already used for the approximation of the state, such that
\[
  \langle \nabla p_{z h} , \nabla q_h \rangle_{L^2(\Omega)} =
  \langle z , q_h \rangle_\Omega \quad \mbox{for all} \; q_h \in V_h
\]
which defines an approximation $\widetilde{A}^{-1} z := p_{z h}$ of
$p_z = A^{-1} z$. From
\[
  \| \nabla p_{zh} \|^2_{L^2(\Omega)} =
  \langle \nabla p_{zh} , \nabla p_{zh} \rangle_{L^2(\Omega)} =
  \langle z , p_{zh} \rangle_\Omega \leq
  \| z \|_{H^{-1}(\Omega)} \| \nabla p_{zh} \|_{L^2(\Omega)}
\]
we immediately conclude boundedness, i.e.,
\[
  \| \widetilde{A}^{-1} z \|_{H^1_0(\Omega)} =
  \| p_{zh}\|_{H^1_0(\Omega)} =
  \| \nabla p_{zh} \|_{L^2(\Omega)} \leq \| z \|_{H^{-1}(\Omega)} \quad
  \mbox{for all} \; z \in H^{-1}(\Omega).
\]
For $z_H \in Z_H$ let $p_{z_Hh} = \widetilde{A}^{-1} z_H \in V_h$ be the
unique solution of the variational formulation
\[
  \langle \nabla p_{z_Hh} , \nabla q_h \rangle_{L^2(\Omega)} =
  \langle z_H , q_h \rangle_{L^2(\Omega)} \quad
  \mbox{for all} \; q_h \in V_h ,
\]
while $p_{z_H} = A^{-1} z_H$ solves
\[
  \langle \nabla p_{z_H} , \nabla q \rangle_{L^2(\Omega)} =
  \langle z_H , q \rangle_{L^2(\Omega)} \quad
  \mbox{for all} \; q \in H^1_0(\Omega) ,
\]
which is the weak formulation of the Poisson equation
$- \Delta p_{z_H} = z_H$ with homogeneous
Dirichlet boundary conditions.
When using standard arguments, i.e., Cea's lemma and the approximation
property of $V_h$, we conclude the error error estimate
\begin{eqnarray*}
  \| \nabla (p_{z_H} - p_{z_Hh}) \|_{L^2(\Omega)}
  & \leq &
           \inf\limits_{q_h \in V_h} \| \nabla (p_{z_H} - q_h) \|_{L^2(\Omega)} 
  \, \leq \,
           c_A \, h \, | p_{z_H} |_{H^2(\Omega)} \\
  & \leq &
  \widetilde{c}_A \, h \, \| \Delta p_{z_H} \|_{L^2(\Omega)} =
  \widetilde{c}_A \, h \, \| z_H \|_{L^2(\Omega)} \leq
  \widetilde{c}_A c_I \, \frac{h}{H} \, \| z_H \|_{H^{-1}(\Omega)},
\end{eqnarray*}
and when using an inverse inequality, e.g., \cite{Steinbach:2008}.
In the case
\begin{equation}\label{mesh condition}
h < \frac{1}{2\widetilde{c}_A c_I} \, H
\end{equation}
we therefore have
\[
  \| \nabla (p_{z_H} - p_{z_Hh}) \|_{L^2(\Omega)} \leq \frac{1}{2} \,
  \| z_H \|_{H^{-1}(\Omega)} .
\]
Hence we can write
\begin{eqnarray*}
  \langle \widetilde{A}^{-1} z_H , z_H \rangle_\Omega
  & = & \langle p_{z_Hh} , z_H \rangle_{L^2(\Omega)} \, = \,
        \langle p_{z_H} , z_H  \rangle_\Omega -
        \langle p_{z_H} - p_{z_Hh} , z_H \rangle_{\Omega} \\
  & \geq &\langle A^{-1} z_H , z_H \rangle_\Omega -
           \| \nabla (p_{z_H} - p_{z_Hh}) \|_{L^2(\Omega)}
           \| z_H \|_{H^{-1}(\Omega)} \geq
           \frac{1}{2} \, \| z_H \|^2_{H^{-1}(\Omega)} .
\end{eqnarray*}
The approximate operator $\widetilde{A}$ is therefore discrete elliptic,
if the mesh condition \eqref{mesh condition} is satisfied.
Although the constants $\widetilde c_A$ and $c_I$ are in general unknown, in our
numerical experiments we have used $h = \frac{1}{4}H$.

Instead of \eqref{FEM control} we now consider the 
variational formulation to find $\widehat{z}_{\varrho H} \in Z_H$ such that
\begin{equation}\label{control pert VF}
  \langle \widetilde{A}^{-1} \widehat{z}_{\varrho H}, \psi_H \rangle_\Omega
  = \langle u_{\varrho h} , \psi_H \rangle_{L^2(\Omega)} \quad
  \mbox{for all} \; \psi_H \in Z_H .
\end{equation}
Unique solvability of \eqref{control pert VF} follows since
$\widetilde{A}^{-1}$ is discrete elliptic. Note that
\eqref{control pert VF} can be written as
a mixed variational formulation to find
$(\widehat{z}_{\varrho H},p_{\widehat{z}_{\varrho H}h})
\in Z_H \times V_h$ such that
\[
  \langle p_{\widehat{z}_{\varrho H}h} , \psi_H \rangle_{L^2(\Omega)} =
  \langle u_{\varrho h} , \psi_H \rangle_{L^2(\Omega)}, \quad
  \langle \nabla p_{\widehat{z}_{\varrho H}h} ,
  \nabla q_h \rangle_{L^2(\Omega)} =
  \langle \widehat{z}_{\varrho H} , q_h \rangle_{L^2(\Omega)}
\]
is satisfied for all $(\psi_H,q_h) \in Z_H \times V_h$.
This formulation is equivalent to the linear system of algebraic equations,
\[
  \hat{M}_h^\top \underline{p} = \hat{M}_h^\top \underline{u}, \quad
  K_h \underline{p} = \hat{M}_h \underline{z},
\]
where in addition to the standard finite element stiffness matrix $K_h$
we have used the mass matrix $\hat M_h$ defined by
\[
  \hat{M}_h[j,\ell] =
  \langle \psi_\ell , \varphi_j \rangle_{L^2(\Omega)}, \quad
  \ell=1,\ldots,N; \; j=1,\ldots, M.
\]
Since the finite element stiffness matrix $K_h$ is invertible,
we can eliminate $\underline{p} = K_h^{-1} \hat{M}_h \underline{z}$ to end
up with the Schur complement system to be solved,
\begin{equation}\label{Schur complement z}
  \hat{M}_h^\top K_h^{-1} \hat{M}_h \underline{z} =
  \hat{M}_h^\top \underline{u} .
\end{equation}
The mesh condition \eqref{mesh condition} not only implies unique
solvability of \eqref{Schur complement z}, but the discrete
ellipticity of $\widetilde{A}^{-1}$ also provides related error estimates,
when applying the Strang lemma, e.g.,
\cite{BrennerScott:1994, Steinbach:2008}. With this
we conclude the final error estimate
\begin{equation}
  \| z_\varrho - \widehat{z}_{\varrho H} \|_{H^{-2}(\Omega)} \leq
  c \, h^s \, \Big( \| \overline{u} \|_{H^s(\Omega)} +
  \| g_\pm \|_{H^2(\Omega)} \Big) \, .
\end{equation}
Note that this estimate remains true when considering control constraints
$f_\pm \in L^2(\Omega)$.

\section{Semi-smooth Newton method}\label{Section:Newton}
In this section we discuss the iterative solution of the discrete
variational inequality \eqref{state LGS}.
 For the solution
of \eqref{nonlinear 1 state} and \eqref{nonlinear 2 state} we can apply
a semi-smooth Newton method which is equivalent to an active set strategy
as given in Algorithm~\ref{alg:active-set},
see \cite{Chen:2001,Hintermueller:2002,Hintermueller:2004,Ito:2008},
and \cite{Steinbach:2014}. The generalization to the iterative
solution of \eqref{nonlinear 1 control} and \eqref{nonlinear 2 control}
is straightforward and will not be discussed here.

\begin{algorithm}
  \caption{Active set algorithm \cite{Hintermueller:2002}}
  \label{alg:active-set}
  \begin{algorithmic}
    \Require{Initial values $\underline{u}^0,\underline\lambda^0$}
    \State (a) $m = 0$
    \State (b) Set
    \[
      y_{+,k}^m = \lambda_k^m+c \, [g_+(x_k)-u_k^m]
      \quad \text{and} \quad y_{-,k}^m = \lambda_k^m+ c \, [g_-(x_k)-u_k^m]
    \]
    \While{stop criterion is not fulfilled}
    \State (i) Set
    \[
      \mathcal{I}^m=\{k:\, y_{+,k}^m\geq 0,\, y_{-,k}^m\leq 0\}, \quad
      \mathcal{A}^m_-=\{k:\, y_{-,k}^m>0\},\quad
      \mathcal{A}^m_+=\{k:\, y_{+,k}^m<0\}
    \]
    \State (ii) Solve
    \[
      (M_h+\varrho K_h)\underline u^{m+1}-\underline{\lambda}^{m+1}=
      \underline{\overline{u}}, \;
      u_k^{m+1}=g_\pm(x_k), \; k \in \mathcal{A}_\pm^m,\;
      \lambda_k^{m+1} =0 , \;  k\in\mathcal{I}^m. 
    \]
    \State (iii) $m = m+1$
    \EndWhile
  \end{algorithmic}
\end{algorithm}

The semi-smooth Newton method successively computes the roots of
$\underline F(\underline{u},\underline{\lambda}) = \underline{0}$ by  
\begin{equation}\label{eq:Newton-algo}
  \begin{pmatrix}
    \underline u^{m+1} \\
    \underline \lambda^{m+1}
  \end{pmatrix}
  =
  \begin{pmatrix}
    \underline u^m \\
    \underline \lambda^m
  \end{pmatrix}
  - \big( D\underline{F}(\underline{u}^m,\underline{\lambda}^m)\big)^{-1}
  \underline{F}(\underline{u}^m,\underline{\lambda}^m),
\end{equation}
with the Jacobian $D\underline F$ given by 
\[
  D\underline F(\underline u,\underline \lambda) =
  \begin{pmatrix}
    M_h + \varrho K_h & -I \\
    c(G_{min}'(\underline u,\underline \lambda) +
    G_{max}'(\underline u,\underline \lambda)) &
    I-(G_{min}'(\underline u,\underline \lambda) +
    G_{max}'(\underline u,\underline \lambda))
  \end{pmatrix} .
\]
The diagonal entries of
\begin{eqnarray*}
  G_{min}'(\underline{u},\underline{\lambda})
  & = & \mbox{diag} \Big( g_{min}' \big( \lambda_k +
        c \, [g_+(x_k)-u_k] \big) \Big), \\
  G_{max}'(\underline{u},\underline{\lambda})
  & = & \mbox{diag} \Big( g_{max}' \big( \lambda_k +
        c \, [g_-(x_k)-u_k] \big) \Big)
\end{eqnarray*}
are the slant derivatives of the functions
$g_{\min}(y) = \min \{ 0, y \}$ and $g_{\max}(y) = \max \{0,y \}$
defined by
\[
  g_{\min}'(y) =
  \begin{cases}
    1,&y< 0, \\
    0,& y\geq 0,
  \end{cases}
  \quad \mbox{and} \quad
  g_{\max}'(y) =
  \begin{cases}
    0,&y\leq 0, \\
    1,& y>0.
  \end{cases}
\]
Rewriting the system \eqref{eq:Newton-algo} gives
\begin{eqnarray}\nonumber
  && \begin{pmatrix}
    M_h + \varrho K_h & -I \\
    c(G_{\min}'(\underline{u}^m,\underline{\lambda}^m) +
    G_{\max}'(\underline{u}^m,\underline{\lambda}^m)) &
    I-(G_{\min}'(\underline{u}^m, \underline{\lambda}^m) +
    G_{\max}'(\underline{u}^m,\underline{\lambda}^m))
  \end{pmatrix}
  \begin{pmatrix}
    \underline{u}^{m}-\underline{u}^{m+1}\\
    \underline{\lambda}^{m} -\underline{\lambda}^{m+1}
  \end{pmatrix} \\[3mm]
  && \hspace*{1cm} = F(\underline u^m,\underline \lambda^m) .
     \label{eq:rewritten-system} 
\end{eqnarray}
From the first line we get 
\[
  (M_h+\varrho K_h) ( \underline{u}^m -\underline{u}^{m+1}) -
  \underline{\lambda}^m + \underline{\lambda}^{m+1} =
  (M_h + \varrho K_h) \underline{u}^m - \underline{\lambda}^m -
  \underline{\overline{u}},
\]
from which we conclude 
\[
  (M_h + \varrho K_h) \underline{u}^{m+1} -
  \underline{\lambda}^{m+1} =
  \underline{\overline{u}} .
\]
With
\[
  y_{+,k}^m := \lambda_k^m + c \, [g_+(x_k)-u_k^m] \quad
  \mbox{and} \quad
  y_{-,k}^m := \lambda_k^m + c \, [g_-(x_k)-u_k^m],
\]
the second line reads, componentwise, 
\begin{eqnarray*}
  && \hspace*{-5mm}
     c [g_{\min}'(y_{+,k}^m) + g_{\max}'(y_{-,k}^m)] (u_k^m-u_k^{m+1})
     + \lambda_k^m - \lambda_k^{m+1} \\
  && - [g_{\min}'(y_{+,k}^m))+g_{\max}'(y_{-,k}^m)]
     (\lambda_k^m-\lambda_k^{m+1}) 
     =
     \lambda_k^m - \min \{ 0,y_{+,k}^m \} - \max \{0,y_{-,k}^m\}. 
\end{eqnarray*} 
We distinguish the following three cases. 
\begin{enumerate}
\item $y_{+,k}^m\geq 0$ and $y_{-,k}^m\leq 0:$ Then,
  $g_{min}'(y_{+,k}^m)= g_{max}'(y_{-,k}^m) =0$, and we compute
  \[
    \lambda_k^{m+1}=0.
  \]
\item $ y_{-,k}^m>0:$ From this we get $\lambda_k^m > c \, [u_k^m-g_-(x_k)]$
  and we compute
  \[
    \lambda_k^m + c \, [ g_+(x_k)-u_k^m] >
    c \, [u_k^m-g_-(x_k)+g_+(x_k)-u_k^m] =
    c \, [g_+(x_k)-g_-(x_k)] > 0,
  \]
  i.e., $y_{+,k}^m>0$. Therefore, $g_{\min}'(y_{+,k}^m)=0$ and
  $g_{\max}'(y_{-,k}^m)=1$, and we get
  \[
    u_k^{m+1} = g_-(x_k).
  \]
\item $ y_{+,k}^m<0:$ Then, as in the second case, we compute
  $y_{-,k}^m<0 $ to get $g_{\min}'(y_{+,k}^m)=1$, $g_{\max}'(y_{-,k}^m)=0$,
  and thus
  \[
    u_k^{m+1}=g_+(x_k).
  \]
\end{enumerate}
Therefore we see, that the iterates of the semi-smooth Newton method
\eqref{eq:Newton-algo} fulfill the active set strategy as given in
Algorithm~\ref{alg:active-set}. 

\section{Numerical results}\label{Section:Results}
For our numerical tests we consider the domain $\Omega = (0,1)^2$ and the
following target functions
$\overline u_i\in \mathcal{C}^\infty(\Omega)\cap H_0^1(\Omega)$, $i=1,2$, 
\begin{align*}\label{eq:function-ubar1}
	\overline u_1(x,y) &= \sin(\pi x)\sin(\pi y),
\end{align*} 
and
\begin{align*}
  \overline u_2(x,y) := \overline u_2(x,y;k) = H_k(x)H_k(y),
  \quad \text{where }
  H_k(s)=\frac{1}{1+e^{-k(s-0.25)}}-\frac{1}{1+e^{-k(s-0.75)}},
\end{align*}
for $k=40$, see Fig.~\ref{fig:adaptive-target-functions}, for which we
can compute $z_i=-\Delta \overline u_i$ analytically. We also consider the
discontinuous target
$\overline u_3:=\lim_{k\to\infty}\overline u_2(\cdot,\cdot;k)\in
H^{1/2-\varepsilon}(\Omega)$,  $\varepsilon>0$, for which we \emph{cannot}
compute the control analytically, given as  
\begin{align*}
	\overline u_3(x,y) = \begin{cases}
		1,& (x,y)\in[0.25,0.75]^2,\\
		0,&\text{else}.
	\end{cases}
\end{align*}

\begin{figure}[htbp]
	\centering
	\begin{tabular}{ccc}
		\begin{subfigure}[b]{0.25\textwidth}
			\centering 
			\includegraphics[width=\textwidth]{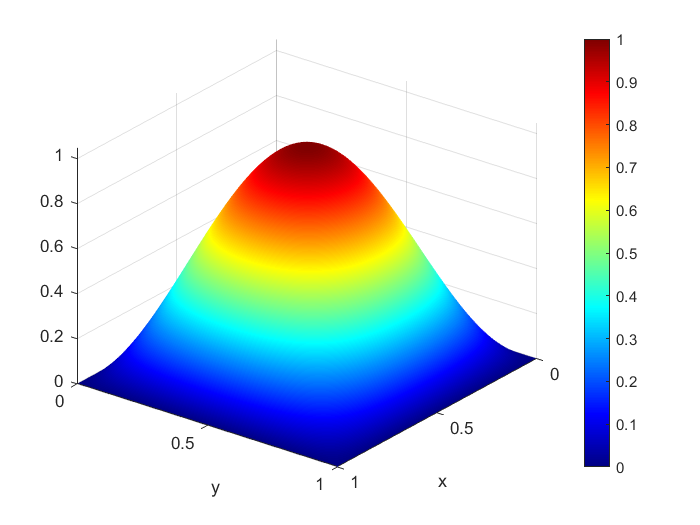}
			\caption{$\overline u_{1}$}
		\end{subfigure}
		&
		\begin{subfigure}[b]{0.25\textwidth}
			\centering 
			\includegraphics[width=\textwidth]{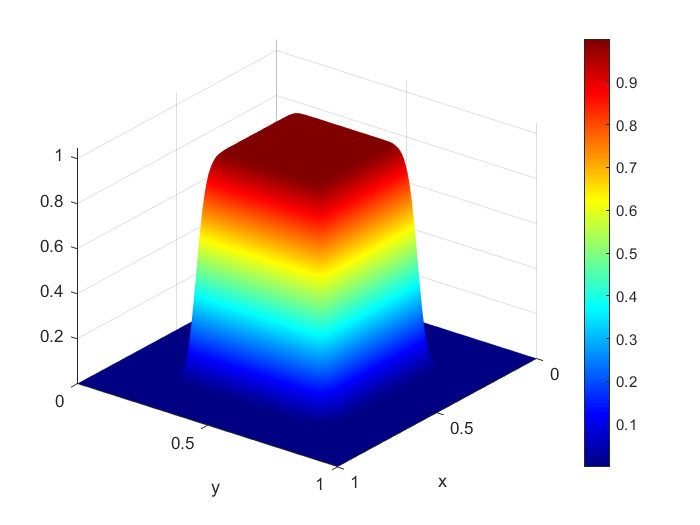}
			\caption{$\overline u_{2}$}
		\end{subfigure}
		&
		\begin{subfigure}[b]{0.25\textwidth}
			\centering
			\includegraphics[width=\textwidth]{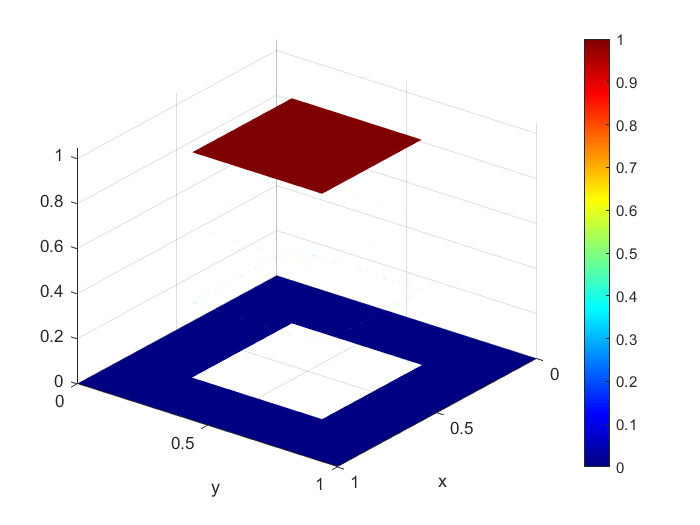}
			\caption{$\overline u_3$}
		\end{subfigure}
		\\
		\begin{subfigure}[b]{0.27\textwidth}
			\centering 
			\includegraphics[width=\textwidth]{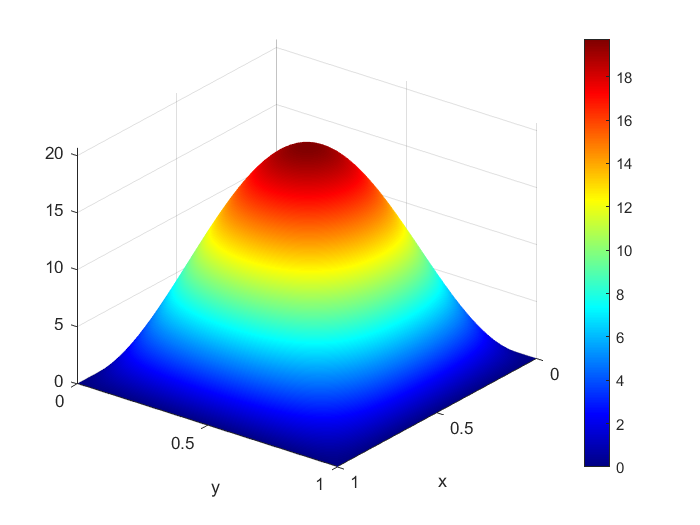}
			\caption{$z_1=-\Delta\overline u_{1}$}
		\end{subfigure}
		&
		\begin{subfigure}[b]{0.27\textwidth}
			\centering 
			\includegraphics[width=\textwidth]{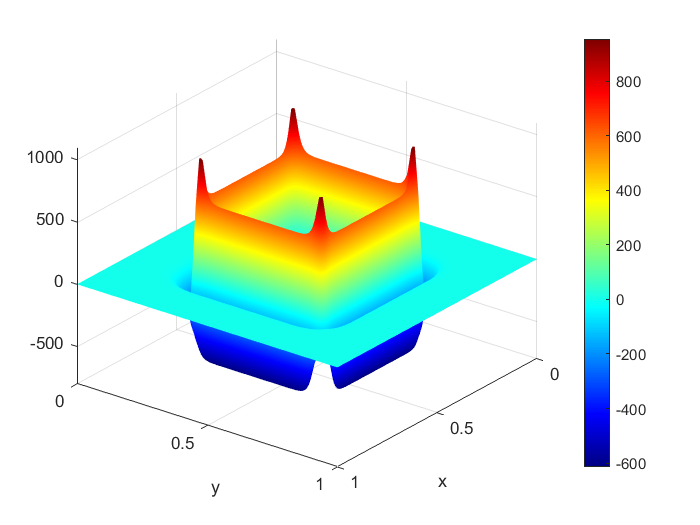}
			\caption{$z_2=-\Delta\overline u_{2}$}
		\end{subfigure}
		&
	\end{tabular}
	\caption{Target functions $\overline u_{i}$ and
		$z_j=-\Delta \overline{u}_j$, $i=1,2,3$, $j=1,2$.}
	\label{fig:adaptive-target-functions}
\end{figure}

\subsection{State constraints}
In order to incorporate constraints on the state, we apply the semi-smooth
Newton algorithm, where we set $\varrho =h^2$, and the initial values
\[
  \underline u^0=(h^2K_h+M_h)^{-1}
  \underline {\overline u}\in\mathbb{R}^M\quad
  \text{and}\quad \underline \lambda^0=\underline 0.
\]
A stopping criterion is then defined using a maximal absolute error in each
node, i.e., we stop if 
\begin{align}\label{stopping}
	\text{tol}:=\max\{\text{tol}_+,\text{tol}_-\}<10^{-5}, 
\end{align}
where
\begin{align*}
  \text{tol}_+:=\max_{\{k:u_k>g_{+}(x_k)\}}|u_k-g_{+}(x_k)|
  \quad \text{and}\quad
  \text{tol}_-:=\max_{\{k: u_k<g_{-}(x_k)\}}| u_k-g_{-}(x_k)|. 
\end{align*}  
After computing the state
$\underline u \leftrightarrow u_{\varrho h}\in K_{sh}$, we can reconstruct
the control $\underline{z} \leftrightarrow z_{\varrho H} \in Z_H$ by solving
the Schur complement system \eqref{Schur complement z}.
In order to ensure stability of the discrete system, we choose $h = H/4$. 
For the targets $\overline u_i$ we consider the upper and lower constraints
$g_{\pm}^{(j)}$ given by 
\[
  g_{-}^{(1)}(x)\equiv 0, \; g_{+}^{(1)}(x)=0.5\cdot\overline u_1(x), \quad
  g_{-}^{(2)}(x)\equiv 0, \; g_{+}^{(2)}(x)=0.5\cdot\overline u_2(x).
\]
The results are depicted in Fig.~\ref{fig:state-reconstruction} and
Fig.~\ref{fig:state-reconstruction-u3}. Note that
$g_+^{(2)}(x,y)\leq\overline u_i(x,y)$ for $i=1,2$ and all $(x,y)\in\Omega$.
Thus, the constrained solutions as shown in
Fig.~\ref{fig:state-reconstruction} (c) and (i) as well as the controls
in (f) and (l) coincide.

\begin{figure}[htbp]
	\centering
	\begin{tabular}{ccc}
		\begin{subfigure}[b]{0.25\textwidth}
			\centering 
			\includegraphics[width=\textwidth]{figures/ubar1.png}
			\caption{$\overline u_1$ }
		\end{subfigure}
		&
		\begin{subfigure}[b]{0.25\textwidth}
			\centering 
			\includegraphics[width=\textwidth]{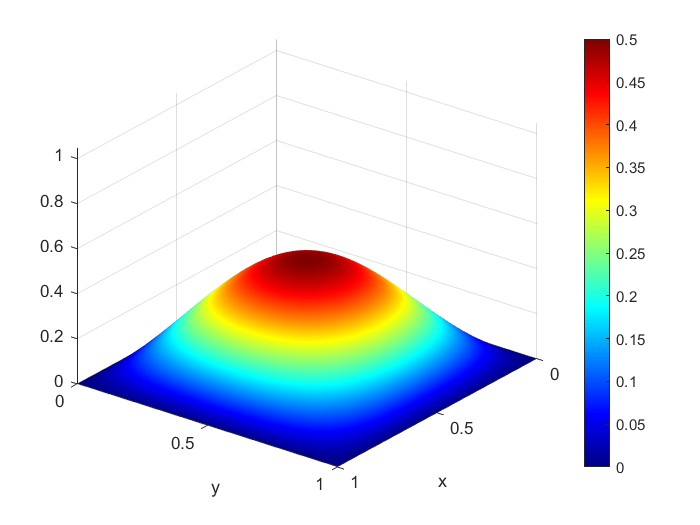}
			\caption{$u_{\varrho h,1}$, $g_\pm^{(1)}$}
		\end{subfigure}
		&
		\begin{subfigure}[b]{0.25\textwidth}
			\centering 
			\includegraphics[width=\textwidth]{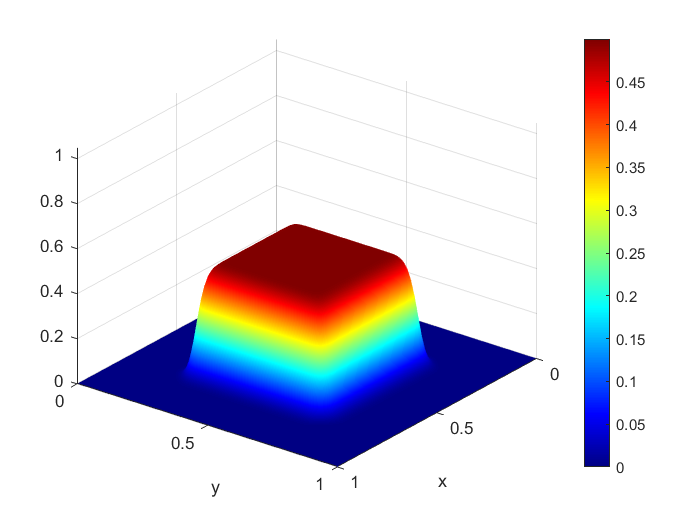}
			\caption{$u_{\varrho h,1}$, $g_\pm^{(2)}$}
		\end{subfigure}
		\\
		\begin{subfigure}[b]{0.25\textwidth}
			\centering 
			\includegraphics[width=\textwidth]{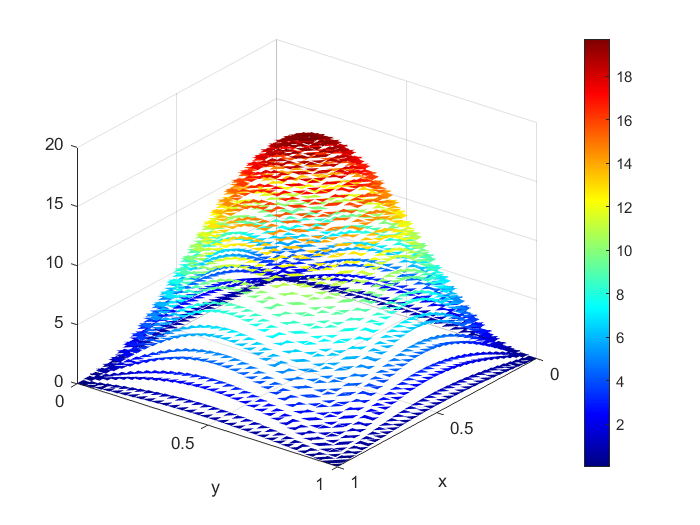}
			\caption{$z_1$ }
		\end{subfigure}
		&
		\begin{subfigure}[b]{0.25\textwidth}
			\centering 
			\includegraphics[width=\textwidth]{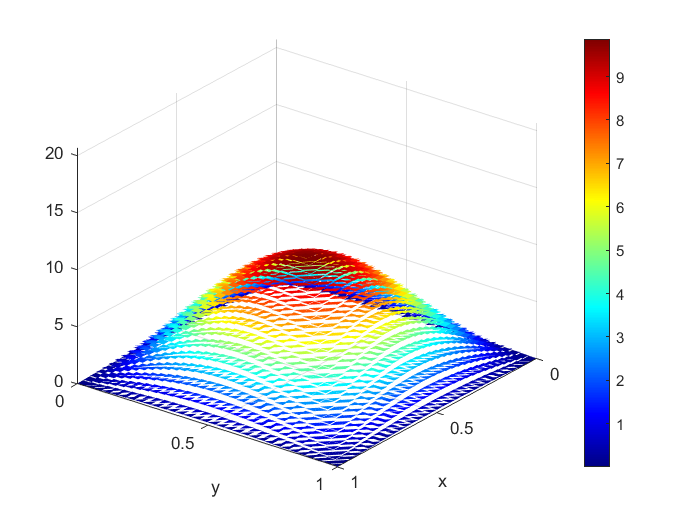}
			\caption{$z_{\varrho H,1}$, $g_\pm^{(1)}$}
		\end{subfigure}
		&
		\begin{subfigure}[b]{0.25\textwidth}
			\centering 
			\includegraphics[width=\textwidth]{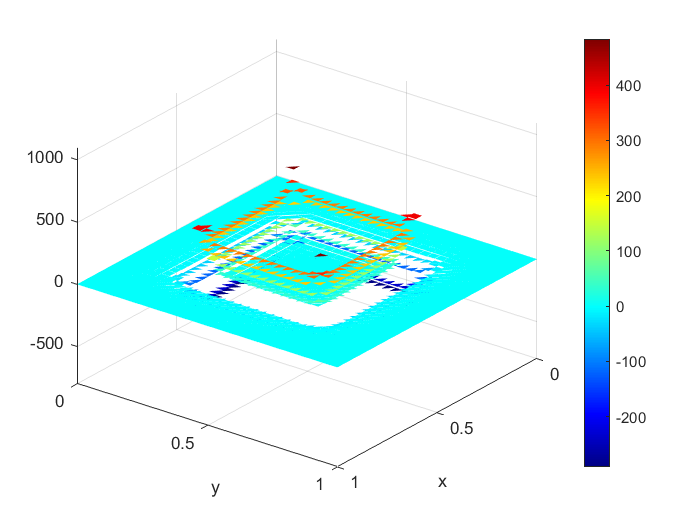}
			\caption{$z_{\varrho H,1}$, $g_\pm^{(2)}$}
		\end{subfigure}
		\\
		\begin{subfigure}[b]{0.25\textwidth}
			\centering 
			\includegraphics[width=\textwidth]{figures/ubar2.png}
			\caption{$\overline u_2$ }
		\end{subfigure}
		&
		\begin{subfigure}[b]{0.25\textwidth}
			\centering 
			\includegraphics[width=\textwidth]{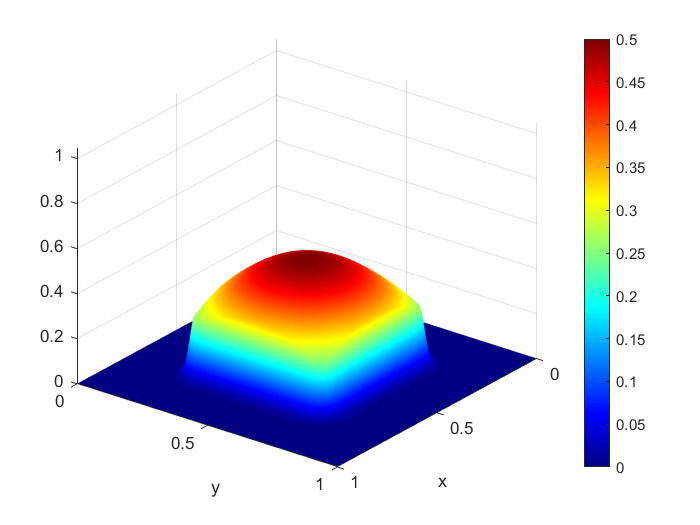}
			\caption{$u_{\varrho h,2}$, $g_\pm^{(1)}$}
		\end{subfigure}
		&
		\begin{subfigure}[b]{0.25\textwidth}
			\centering 
			\includegraphics[width=\textwidth]{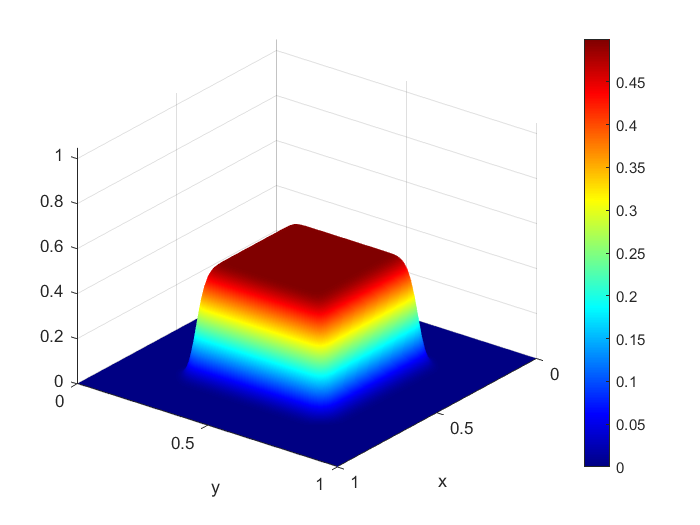}
			\caption{$u_{\varrho h,2}$, $g_\pm^{(2)}$}
		\end{subfigure}
		\\
		\begin{subfigure}[b]{0.25\textwidth}
			\centering 
			\includegraphics[width=\textwidth]{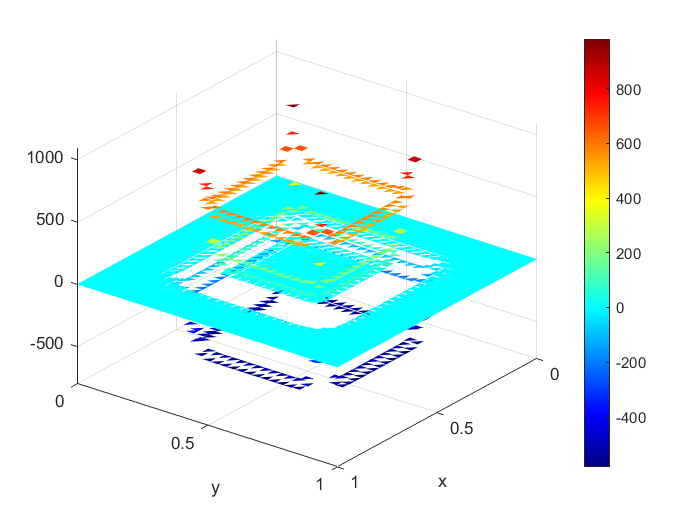}
			\caption{$z_2$ }
		\end{subfigure}
		&
		\begin{subfigure}[b]{0.25\textwidth}
			\centering 
			\includegraphics[width=\textwidth]{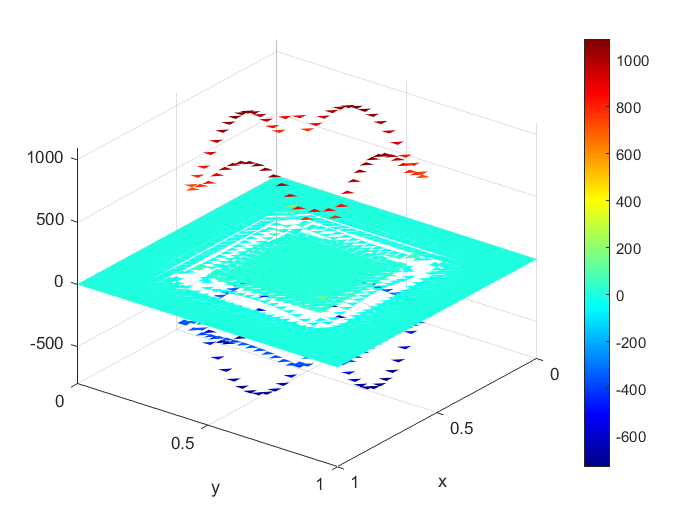}
			\caption{$z_{\varrho H,2}$, $g_\pm^{(1)}$}
		\end{subfigure}
		&
		\begin{subfigure}[b]{0.25\textwidth}
			\centering 
			\includegraphics[width=\textwidth]{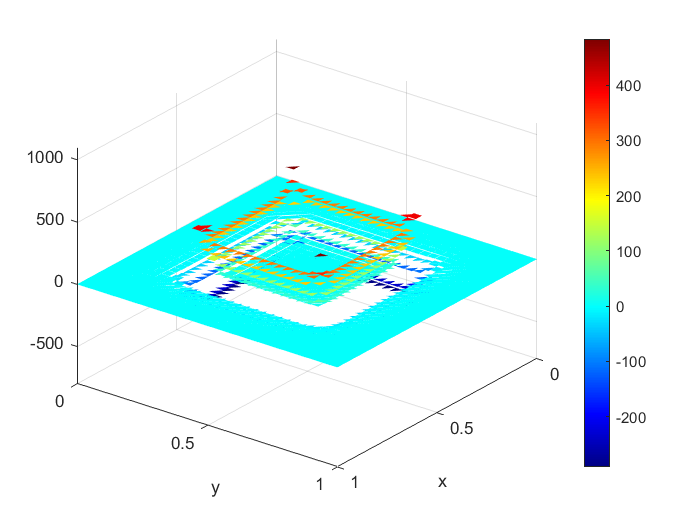}
			\caption{$z_{\varrho H,2}$, $g_\pm^{(2)}$}
		\end{subfigure}
	\end{tabular}
	\caption{Targets $\overline u_i$ and unconstrained controls $z_i$, $i=1,2$, computed constrained states $u_{\varrho h,i}$ on a mesh with $N=32768$ elements and $M=16129$ DoFs with constraints $g_\pm^{(j)}$, and reconstruction of the controls $z_{\varrho H,i}$ on a mesh with $N_H=2048$ elements. }
	\label{fig:state-reconstruction}
\end{figure}

\subsection{Control constraints}
In order to incorporate contraints on the control, we apply the 
semi-smooth Newton algorithm, where we set $\varrho =h^2$, and the
initial values
\[
  \underline u^0=(h^2K_h+M_h)^{-1}\underline {\overline u}\in\mathbb{R}^M\quad
  \text{and}\quad \underline{w}^0=\underline 0.
\]
Again we apply the stopping criteria \eqref{stopping} but now we use
\begin{align*}
  \text{tol}_+:=\max_{\{k:(K_h\underline u)_k>f_{+,k}\}}
  |(K_h\underline u)_k-f_{+,k}| \quad \text{and}\quad
  \text{tol}_-:=\max_{\{k:(K_h\underline u)_k<f_{-,k}\}}
  |(K_h\underline u)_k-f_{-,k}|. 
\end{align*}  
For the upper and lower constraints on the control we consider the functions
$f_{\pm}^{(j)}$ given by 
\begin{align*}
  f_{-}^{(1)}(x)\equiv 0\quad
  &\text{and}\quad  f_{+}^{(1)}(x)=0.5\cdot z_1(x)\\
  f_{-}^{(2)}(x)\equiv 0\quad
  &\text{and}\quad  f_{+}^{(2)}(x)=\min\{z_1(x),10\}\\
  f_{-}^{(3)}(x)= \max\{\min\{z_2(x),0\},-500\}\quad
  &\text{and}\quad  f_{+}^{(3)}(x)=\min\{\max\{z_2(x),0\},500\}\\
  f_{-}^{(4)}(x)\equiv 0\quad
  &\text{and}\quad  f_{+}^{(4)}(x)=\min\{\max\{z_2(x),0\},1000\}\\
  f_{-}^{(5)}(x)\equiv 0\quad
  &\text{and}\quad  f_{+}^{(5)}(x)=4\cdot \min\{\max\{z_2(x),0\},250\}.
\end{align*} 
The results are depicted in Fig.~\ref{fig:control-reconstruction} and
Fig.~\ref{fig:state-reconstruction-u3}. Note, that for $\varrho=0$ the
control for $\overline u_3$ is given by
$z_3=-\Delta \overline u_3 \in H^{-3/2-\varepsilon}(\Omega)$ and thus can
only be seen in a distributional sense. If we compute the control on a
sufficiently fine mesh, this behaviour is resembled and the control explodes
only in some points. Thus, we also give the reconstruction on a coarser mesh
in Fig.~\ref{fig:state-reconstruction-u3} and with suitable constraints,
which still gives a meaningful control.

\begin{figure}[htbp]
	\centering
	\begin{subfigure}[b]{0.25\textwidth}
		\centering 
		\includegraphics[width=\textwidth]{figures/ubar1.png}
		\caption{$\overline u_1$ }
	\end{subfigure}
	\begin{subfigure}[b]{0.25\textwidth}
		\centering 
		\includegraphics[width=\textwidth]{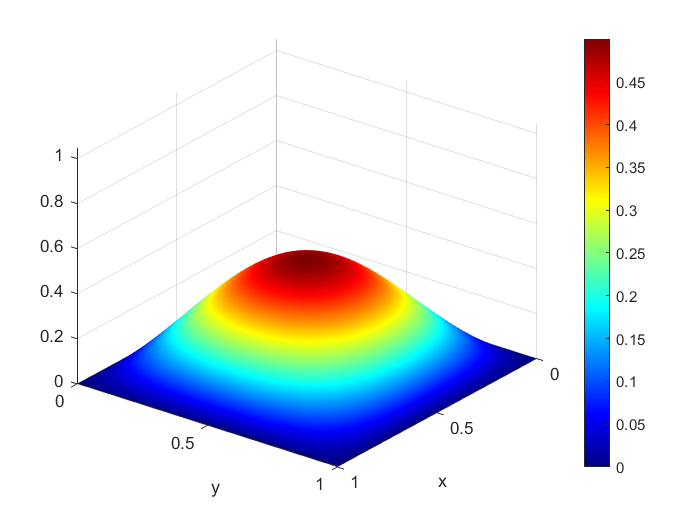}
		\caption{$u_{\varrho h,1}$, $f_\pm^{(1)}$}
	\end{subfigure}
	\begin{subfigure}[b]{0.25\textwidth}
		\centering 
		\includegraphics[width=\textwidth]{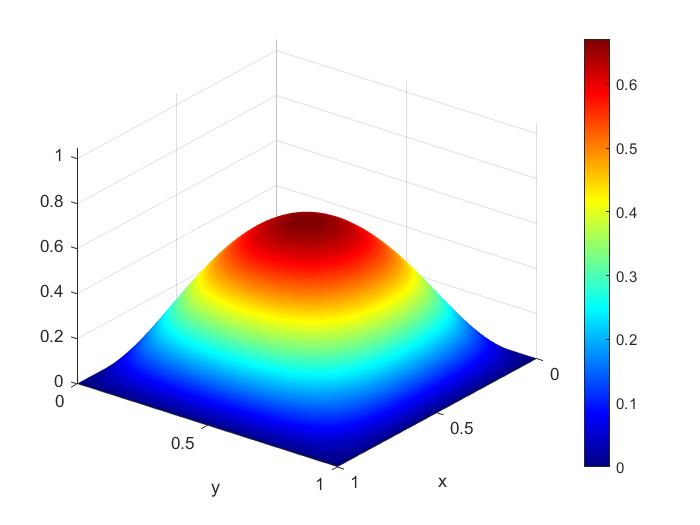}
		\caption{$u_{\varrho h,1}$, $f_\pm^{(2)}$}
	\end{subfigure}
	\\
	\begin{subfigure}[b]{0.25\textwidth}
		\centering 
		\includegraphics[width=\textwidth]{figures/z1_triags.png}
		\caption{$z_1$ }
	\end{subfigure}
	\begin{subfigure}[b]{0.25\textwidth}
		\centering 
		\includegraphics[width=\textwidth]{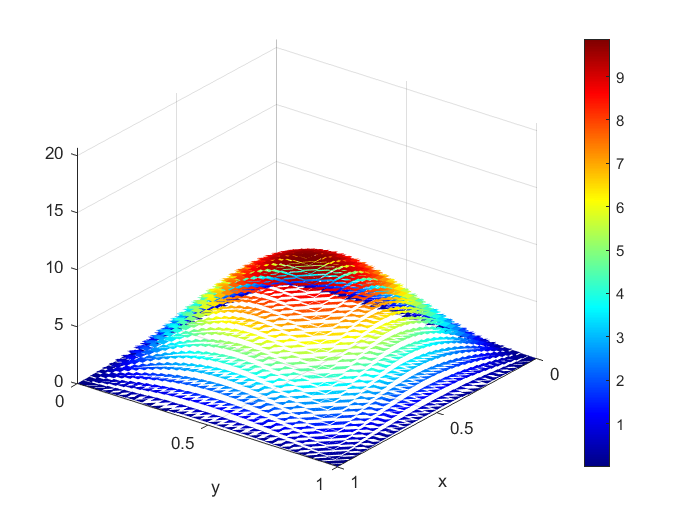}
		\caption{$z_{\varrho H,1}$, $f_\pm^{(1)}$}
	\end{subfigure}
	\begin{subfigure}[b]{0.25\textwidth}
		\centering 
		\includegraphics[width=\textwidth]{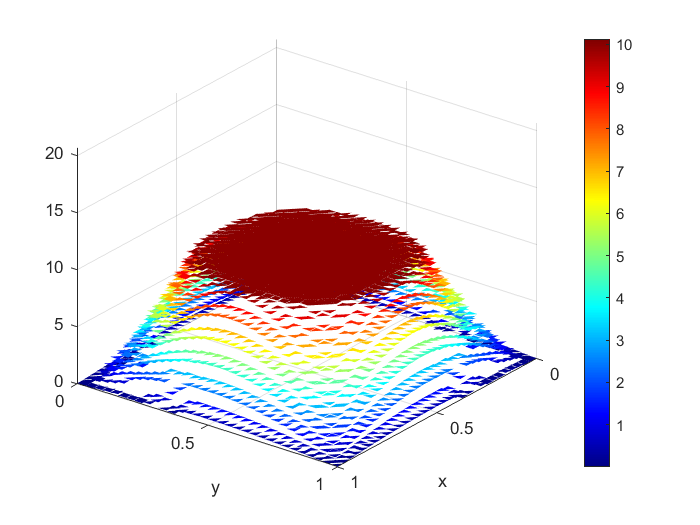}
		\caption{$z_{\varrho H,1}$, $f_\pm^{(2)}$}
	\end{subfigure}
	\\
	\begin{subfigure}[b]{0.25\textwidth}
		\centering 
		\includegraphics[width=\textwidth]{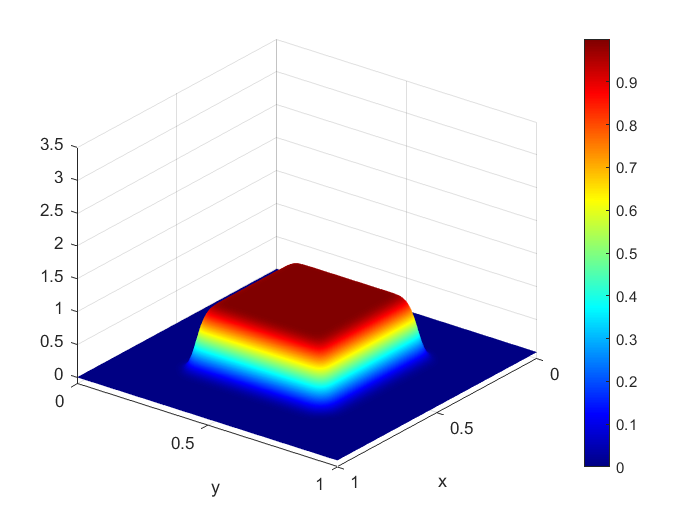}
		\caption{$\overline u_2$ }
	\end{subfigure}
	\begin{subfigure}[b]{0.25\textwidth}
		\centering 
		\includegraphics[width=\textwidth]{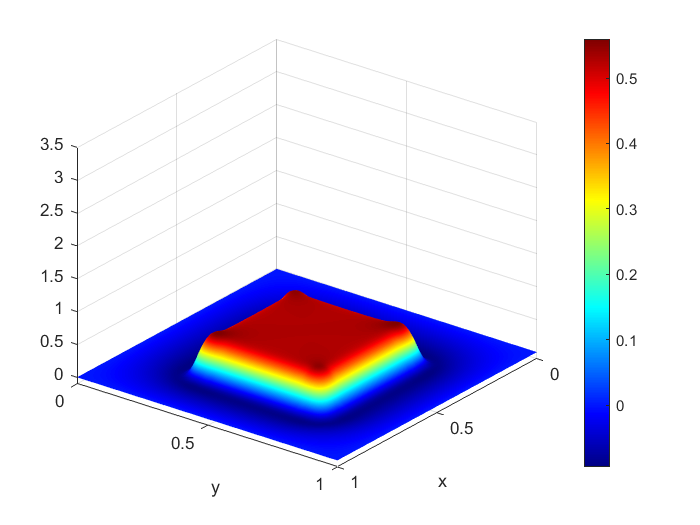}
		\caption{$u_{\varrho h,2}$, $f_\pm^{(3)}$}
	\end{subfigure}
	\begin{subfigure}[b]{0.25\textwidth}
		\centering 
		\includegraphics[width=\textwidth]{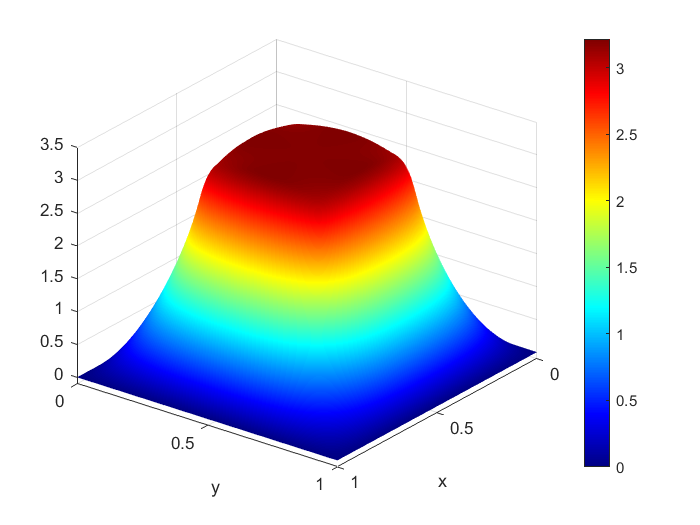}
		\caption{$u_{\varrho h,2}$, $f_\pm^{(4)}$}
	\end{subfigure}
	\\
	\begin{subfigure}[b]{0.25\textwidth}
		\centering 
		\includegraphics[width=\textwidth]{figures/z2_triags.png}
		\caption{$z_2$ }
	\end{subfigure}
	\begin{subfigure}[b]{0.25\textwidth}
		\centering 
		\includegraphics[width=\textwidth]{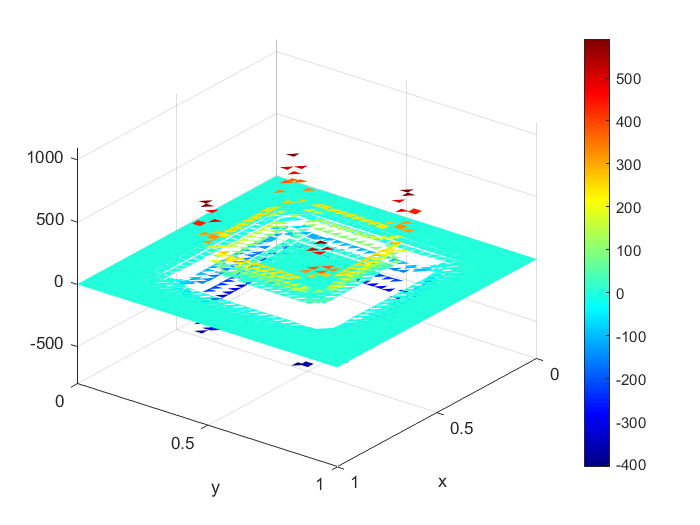}
		\caption{$z_{\varrho H,2}$, $f_\pm^{(3)}$}
	\end{subfigure}
	\begin{subfigure}[b]{0.25\textwidth}
		\centering 
		\includegraphics[width=\textwidth]{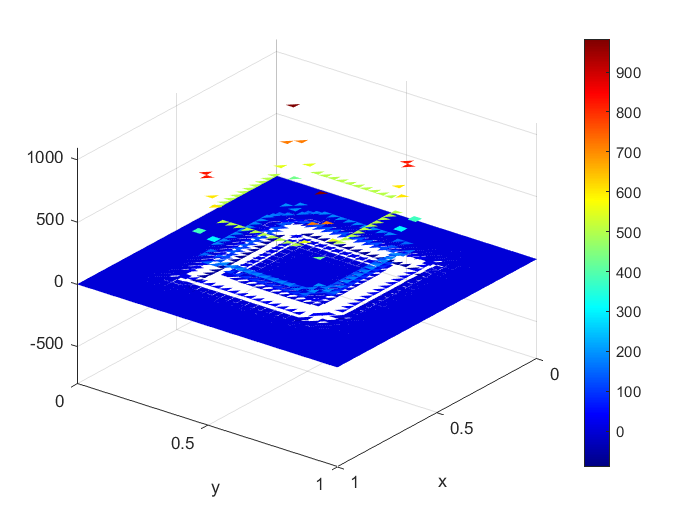}
		\caption{$z_{\varrho H,2}$, $f_\pm^{(4)}$}
	\end{subfigure}
	\caption{Targets $\overline u_i$ and unconstrained controls $z_i$, $i=1,2$, computed states $u_{\varrho h,i}$ on a mesh with $N=32768$ elements and $M=16129$ DoFs and reconstruction of the constrained controls $z_{\varrho H,i}$, with constraints $f_{\pm}^{(j)}$, on a mesh with $N_H=2048$ elements. }
	\label{fig:control-reconstruction}
\end{figure}

\begin{figure}[htbp]
	\centering
	\begin{tabular}{ccc}
		\begin{subfigure}[b]{0.24\textwidth}
			\centering 
			\includegraphics[width=\textwidth]{figures/ubar3.png}
			\caption{$\overline u_{3}$}
		\end{subfigure}
		&
		\begin{subfigure}[b]{0.24\textwidth}
			\centering 
			\includegraphics[width=\textwidth]{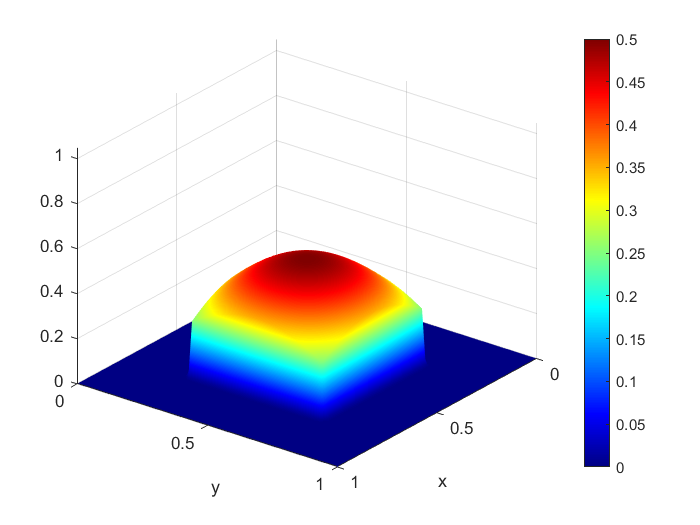}
			\caption{$u_{\varrho h,3}$, $g_\pm^{(1)}$}
		\end{subfigure}
		&
		\begin{subfigure}[b]{0.24\textwidth}
			\centering 
			\includegraphics[width=\textwidth]{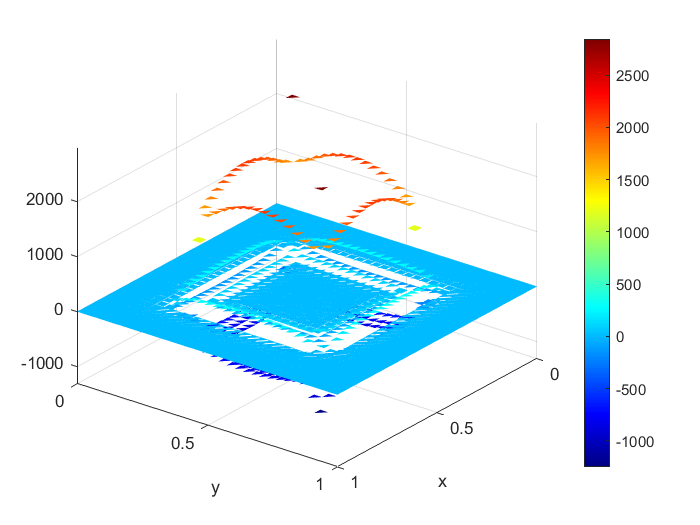}
			\caption{$z_{\varrho H,3}$, $g_\pm^{(1)}$}
		\end{subfigure}
		\\
		&
		\begin{subfigure}[b]{0.24\textwidth}
			\centering 
			\includegraphics[width=\textwidth]{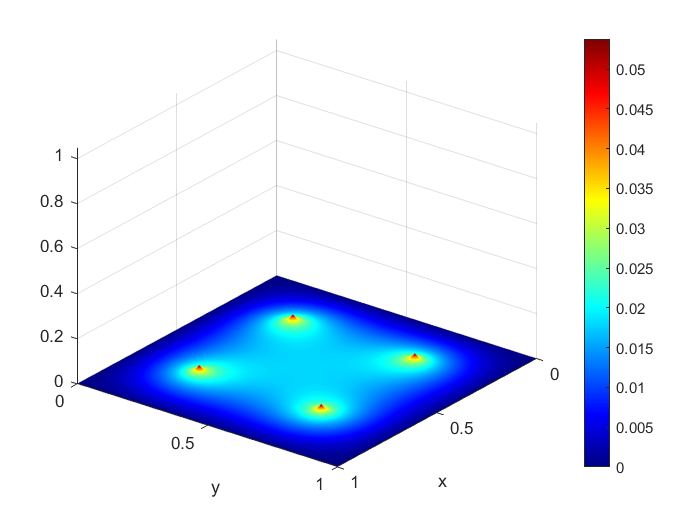}
			\caption{$u_{\varrho h,3}$, $f_\pm^{(4)}$}
		\end{subfigure}
		&
		\begin{subfigure}[b]{0.24\textwidth}
			\centering 
			\includegraphics[width=\textwidth]{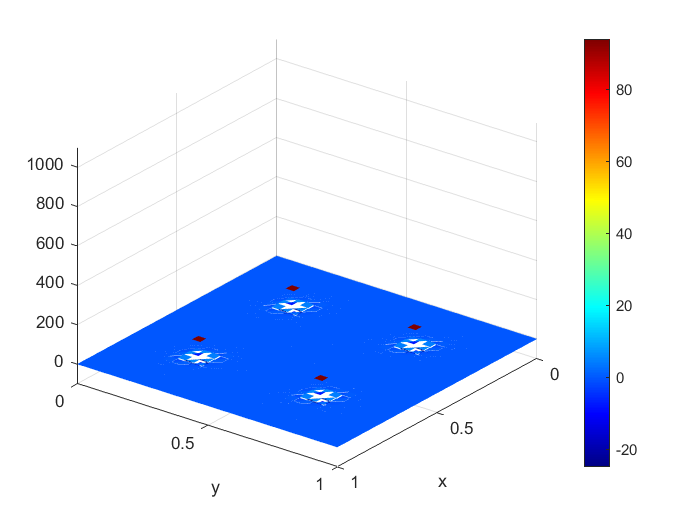}
			\caption{$z_{\varrho H,3}$, $f_\pm^{(4)}$}
		\end{subfigure}
		\\
		\begin{subfigure}[b]{0.24\textwidth}
			\centering 
			\includegraphics[width=\textwidth]{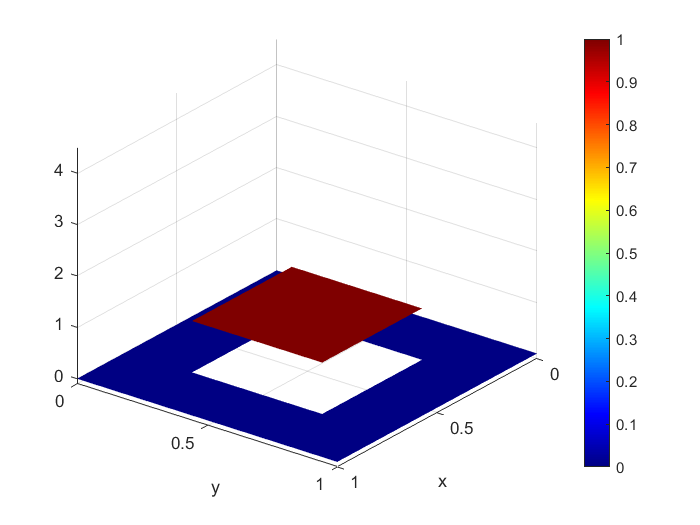}
			\caption{$\overline u_{3}$}
		\end{subfigure}
		&
		\begin{subfigure}[b]{0.24\textwidth}
			\centering 
			\includegraphics[width=\textwidth]{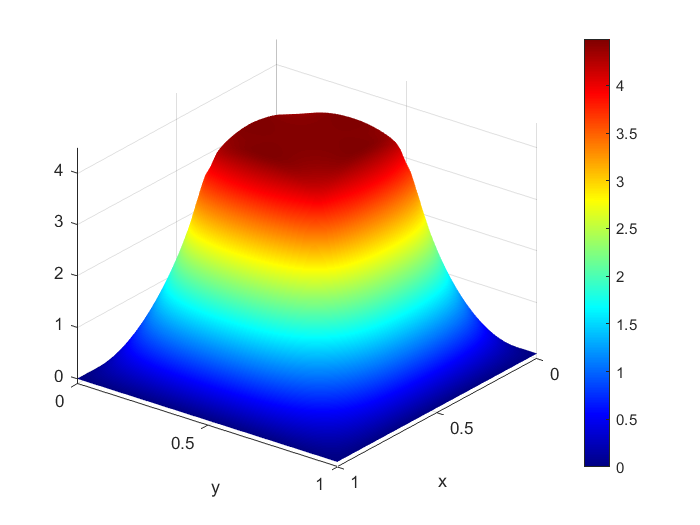}
			\caption{$u_{\varrho h,3}$, $f_\pm^{(5)}$}
		\end{subfigure}
		&
		\begin{subfigure}[b]{0.24\textwidth}
			\centering 
			\includegraphics[width=\textwidth]{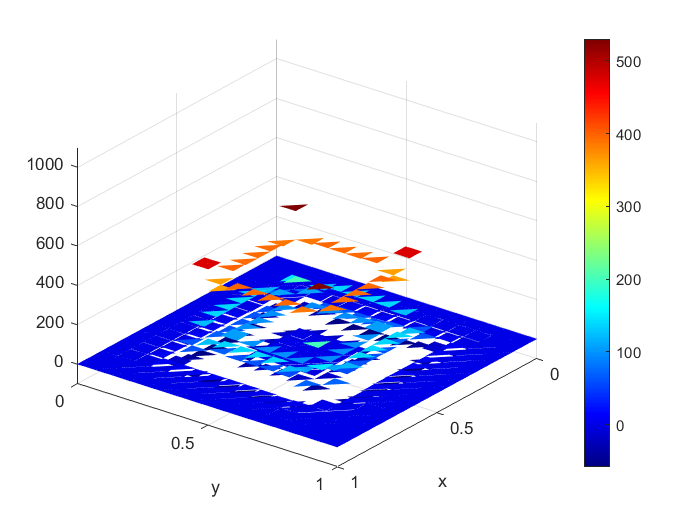}
			\caption{$z_{\varrho H,3}$, $f_\pm^{(5)}$}
		\end{subfigure}	
	\end{tabular}
\caption{Computed state $u_{\varrho h,3}$ for state constaints $g_{\pm}^{(1)}$, control constaints $f_\pm^{(4)}$ and $f_\pm^{(5)}$ on a mesh with $N=32768$ elements and $M=16129$ DoFs adn reconstruction of the control $z_{\varrho H,3}$ on a mesh with $N_H=2048$ elements ($1^{\text{st}}$ and $2^{\text{nd}}$ row) and with $N_H=512$ ($3^{\text{rd}}$ row).} 
\label{fig:state-reconstruction-u3}
\end{figure}

\section{Conclusions}
In this paper, we have described and analyzed
state and control constraints when considering elliptic distributed
optimal control problems with energy regularization. We have proven
optimal error estimates with respect to both the regularization
parameter $\varrho$, and the finite element mesh width $h$.
While for the solution of the nonlinear system we have used a
semi-smooth Newton method, the design of an overall efficient
iterative solution method including preconditioning
was not within the scope of this paper. This approach can be extended
to optimal control problems in three space dimensions, and to more
involved applications. Moreover, following existing work for
unconstrained optimal control problems subject to time dependent
problems such as the heat and the wave equation, we can include state
and control constraints also in these cases.

\section*{Acknowledgments}
This work has been partially supported by the Austrian
Science Fund (FWF) under the Grant Collaborative Research Center
TRR361/F90: CREATOR Computational Electric Machine Laboratory.

\bibliographystyle{abbrv}
\bibliography{constraint}

\end{document}